\documentclass[11pt,leqno]{article}
\usepackage{amsmath,amsthm,amsfonts,amssymb}

\newcommand{\N}{{\mathbb N}}
\newcommand{\Z}{{\mathbb Z}}
\newcommand{\R}{{\mathbb R}}
\newcommand{\C}{{\mathbb C}}
\newcommand{\bD}{{\mathbb D}}
\newcommand{\E}{{\mathbb E}}
\newcommand{\bL}{{\mathbb L}}
\newcommand{\bB}{{\mathbb B}}
\newcommand{\eps}{{\varepsilon}}
\newcommand{\tauetabar}{{(\underline{\tau},\underline{\eta})}}
\newcommand{\utau}{\underline{\tau}}
\newcommand{\ueta}{\underline{\eta}}

\newcommand{\bA}{{\mathbb A}}

\newcommand\cA{{\cal  A}}
\newcommand\cB{{\cal  B}}
\newcommand\cD{{\cal  D}}
\newcommand\cH{{\cal  H}}
\newcommand\cV{{\cal  V}}

\newcommand\cI{{\cal  I}}

\newcommand\cG{{\cal  G}}

\newcommand\cE{{\cal  E}}
\newcommand\cF{{\cal  F}}

\newcommand\cO{{\cal O}}

\newcommand\cU{{\mathcal U}}

\newcommand\cL{{\mathcal L}}

\newcommand{\opeg}{{\text{\rm Op}^{\varepsilon,\gamma}}}
\newcommand{\bfS}{{\bf S}}

\newcommand{\uzeta}{\underline{\zeta}}

\newcommand{\ua}{\underline{a}}

\newcommand{\X}{{\xi+\frac{k\, \beta}{\eps}}}

\newcommand\adots{\mathinner{\mkern2mu\raise1pt\hbox{.}
\mkern3mu\raise4pt\hbox{.}\mkern1mu\raise7pt\hbox{.}}}

\newcommand{\re}{{\rm Re }\, }

\newtheorem{theo}{Theorem}[section]
\newtheorem{prop}[theo]{Proposition}
\newtheorem{cor}[theo]{Corollary}
\newtheorem{lem}[theo]{Lemma}
\newtheorem{defn}[theo]{Definition}

\newtheorem{rem}[theo]{Remark}

\newtheorem{nota}[theo]{Notations}
\newtheorem{assumption}[theo]{Assumption}

\numberwithin{equation}{section}

 \title{Amplification of pulses in nonlinear geometric optics}

\author{Jean-Fran\c{c}ois {\sc Coulombel}\thanks{CNRS and Universit\'e de Nantes, Laboratoire de math\'ematiques 
Jean Leray (UMR CNRS 6629), 2 rue de la Houssini\`ere, BP 92208, 44322 Nantes Cedex 3, France. Email: 
{\tt jean-francois.coulombel@univ-nantes.fr}.},
Mark {\sc Williams}\thanks{University of North Carolina, Mathematics Department, CB 3250, Phillips Hall, 
Chapel Hill, NC 27599. USA. Email: {\tt williams@email.unc.edu}. Research of M.W. was partially supported by 
NSF grants number DMS-0701201 and DMS-1001616.}}

\begin{document}

\maketitle

\begin{abstract}
In this companion paper to our study of amplification of wavetrains \cite{CGW3}, we study weakly stable semilinear 
hyperbolic boundary value problems with pulse data. Here weak stability means that exponentially growing modes 
are absent, but the so-called uniform Lopatinskii condition fails at some boundary frequency in the hyperbolic region. 
As a consequence of this degeneracy there is again an amplification phenomenon: outgoing pulses of amplitude 
$O(\eps^2)$ and wavelength $\eps$ give rise to reflected pulses of amplitude $O(\eps)$, so the overall solution 
has amplitude $O(\eps)$. Moreover, the reflecting pulses emanate from a radiating pulse that propagates in the 
boundary along a characteristic of the Lopatinskii determinant.

In the case of $N\times N$ systems considered here, a single outgoing pulse produces on reflection a family of 
incoming pulses traveling at different group velocities. Unlike wavetrains, pulses do not interact to produce resonances 
that affect the leading order profiles. However, pulse interactions do affect lower order profiles and so these interactions 
have to be estimated carefully in the error analysis. Whereas the error analysis in the wavetrain case dealt with small 
divisor problems by approximating periodic profiles by trigonometric polynomials (which amounts to using a high 
frequency cutoff), in the pulse case we approximate decaying profiles with nonzero moments by profiles with zero 
moments (a low frequency cutoff). Unlike the wavetrain case, we are now able to obtain a rate of convergence in 
the limit describing convergence of approximate to exact solutions.
\end{abstract}

\tableofcontents

\section{Introduction}\label{intro}

\emph{\quad}In this paper we study the reflection of pulses in weakly stable semilinear hyperbolic boundary value 
problems. The problems are weakly stable in the sense that exponentially growing modes are absent, but the 
uniform Lopatinskii condition fails at a boundary frequency $\beta$ in the hyperbolic region\footnote{See Definition 
\ref{def1} and Assumption \ref{assumption3} for precise statements.} $\cH$. As a consequence of this degeneracy 
in the boundary conditions, there is an amplification phenomenon: boundary data of amplitude $\eps$ in the problem 
\eqref{2} below gives rise to a response of amplitude $O(1)$.

On $\overline {\mathbb{R}}^{d+1}_+ = \{x=(x',x_d)=(t,y,x_d)=(t,x''):x_d\geq 0\}$, consider the $N\times N$ semilinear 
hyperbolic boundary problem  for $v=v_\eps(x)$, where $\eps>0$\footnote{We usually suppress the subscript $\eps$.}:
\begin{align}\label{0}
\begin{split}
&(a)\;L(\partial)v+f_0(v)=0\\
&(b)\;\phi(v)=\eps^2 \, G(x',\frac{x'\cdot\beta}{\eps})\text{ on }x_d=0\\
&(c)\;v=0 \text{ and }G=0\text{ in }t<0,
\end{split}
\end{align}
where $L(\partial):=\partial_t+\sum^{d}_{j=1}B_j\partial_j$, the matrix $B_d$ is invertible, and both $f_0(v)$ and 
$\phi(v)$ vanish at $v=0$. The function $G(x',\theta_0)$ is assumed to have polynomial decay in $\theta_0$, and 
the frequency $\beta \in \R^d \setminus \{0\}$ is taken to be a boundary frequency at which the uniform Lopatinskii 
condition fails. A consequence of this failure is that the choice of the factor $\eps^2$ in \eqref{0}(b) corresponds to 
the weakly nonlinear regime for this problem. The leading order profile is coupled to the next order profile in the 
nonlinear system \eqref{14b} derived below.

Before proceeding we write the problem in an equivalent form that is better adapted to the boundary.
After multiplying the  first equation by  $(B_d)^{-1}$ we obtain
\begin{align*}
\begin{split}
&\tilde L(\partial)v+f(v)=0\\
&\phi(v)=\eps^2 \, G(x',\frac{x'\cdot\beta}{\eps})\text{ on }x_d=0\\
&v=0 \text{ and }G=0\text{ in }t<0,
\end{split}
\end{align*}
where we have set
\begin{align*}
\tilde L(\partial):=\partial_d+\sum^{d-1}_{j=0}A_j \, \partial_j \text{ with }A_j :=B_d^{-1}B_j 
\text{ for } j=0,\dots,d-1 \, ,\quad \partial_0 := \partial_t \, .
 \end{align*}
Setting $v=\eps u$  and writing  $f(v)=D(v)v$, $\phi(v)=\psi(v)v$, we get the problem for $u=u_\eps(x)$
\begin{align}\label{2}
\begin{split}
&(a)\;\tilde L(\partial)u+D(\eps u)u=0\\
&(b)\;\psi(\eps u)u=\eps \, G(x',\frac{x'\cdot\beta}{\eps})\text{ on }x_d=0\\
&(c)\;u=0 \text{ in }t<0.
\end{split}
\end{align}

For this problem we pose the two basic questions of rigorous nonlinear geometric optics:\\

(1)\;Does an exact solution $u_\eps$ of \eqref{2} exist for $\eps\in (0,1]$ on a fixed time interval $[0,T_0]$ 
independent of $\eps$?

(2)\;  Suppose the answer to the first question is yes.  If we let $u^{app}_\eps$ denote an approximate solution 
on $[0,T_0]$ constructed by the methods of nonlinear geometric optics (that is, solving eikonal equations for 
phases and suitable transport equations for profiles), how well does $u^{app}_\eps$ approximate $u_\eps$ for 
$\eps$ small? For example, is it true that
\begin{align}\label{2z}
\lim_{\eps\to 0}|u_\eps-u^{app}_\eps|_{L^\infty}\to 0 \, ?
\end{align}

The amplification phenomenon was studied in a formal way for several different quasilinear problems in 
\cite{AM,MA,MR}. In \cite{MR} amplification was studied in connection with Mach stem formation in reacting 
shock fronts, while \cite{AM} explored a connection to the formation of instabilities in compressible vortex sheets. 
Both papers derived equations for profiles using an ansatz that exhibited amplification; however, neither of the 
two questions posed above were addressed. The first rigorous amplification results were proved in \cite{CG} 
for \emph{linear} wavetrains. That paper provided positive answers to the above questions by making use of 
approximate solutions of high order, and showed in particular that the limit \eqref{2z} holds.

An amplification result for wavetrains in weakly stable \emph{semilinear} problems was proved in \cite{CGW3}, 
and in this paper we prove the analogous result for pulses. Our approximate solution has the form
\begin{equation}
\label{a10z}
u^{app}_\eps(x)= \cU^0 ( x,\theta_0,\xi_d)|_{\theta_0=\frac{\phi_0}{\eps},\xi_d=\frac{x_d}{\eps}},
\end{equation}
where\footnote{The notation in \eqref{a10z} and \eqref{a10y} is explained in sections \ref{assumptions} and \ref{mr}.}
\begin{align}\label{a10y}
\cU^0 \left( x,\theta_0,\xi_d \right) =\sum_{m\in\cI} \, \sum^{\nu_{k_m}}_{k=1}
\sigma_{m,k}\left(x,\theta_0+\omega_m\xi_d\right) \, r_{m,k}.
\end{align}
Here each pulse profile $\sigma_{m,k}(x,\theta)$ is polynomially decaying in $\theta\in\R$, rather than periodic in 
$\theta$ as in the wavetrain case. From the explicit and rather simple equations satisfied by the profiles appearing 
in $\cU^0$, we read off the main amplification phenomenon for solutions of \eqref{2}: in $t>0$ pulses of amplitude 
$O(1)$, which emanate from a pulse in the boundary that propagates along a characteristic of the Lopatinskii 
determinant, radiate into the interior at distinct group velocities determined by $\beta$.

In the case of $N\times N$ systems considered here, a single outgoing pulse produces on reflection a family of 
incoming pulses traveling at different group velocities. Unlike wavetrains, pulses do not interact to produce resonances 
that affect the leading order profiles. One effect of this is that we do not need to use Nash-Moser iteration to construct 
the leading pulse profiles as we did in the wavetrain case. The construction of exact solutions is again based on 
the study of an associated singular system \eqref{15p}. This system is solved by a Nash-Moser iteration argument 
based on a tame estimate for the linearized singular system \eqref{i33a}. The proof of that estimate and the 
Nash-Moser iteration can be carried over {\it verbatim} from the corresponding argument in \cite{CGW3} for 
wavetrains\footnote{Thus, we do not repeat those (lengthy) arguments here.}. The estimate for the linearized 
singular system also plays a key role in the error analysis here.

The main differences between the analysis of this paper and that of \cite{CGW3} occur in the study of the profile 
equations and in the error analysis. Although pulses do not produce resonances, pulse interactions do affect lower 
order profiles and so these interactions have to be estimated carefully in the error analysis. Unlike the wavetrain 
case, we are now able to obtain a rate of convergence in the limit \eqref{2z}. Whereas the error analysis in the 
wavetrain case dealt with small divisor problems by approximating periodic profiles by trigonometric polynomials 
(which amounts to using a high frequency cutoff), in the pulse case we approximate decaying profiles with nonzero 
moments by profiles with zero moments (using a low frequency cutoff)\footnote{We borrow the idea of introducing 
low frequency cutoffs from \cite{AR2} where single phase pulses in free space are studied.}. We provide a more 
detailed summary of the argument after stating our main assumptions.

\subsection{Assumptions}
\label{assumptions}

We make the following hyperbolicity assumption on the system \eqref{0}:

\begin{assumption}
\label{assumption1}
There exist an integer $q \ge 1$, some real functions $\lambda_1,\dots,\lambda_q$ that are analytic on $\R^d
\setminus \{ 0 \}$ and homogeneous of degree $1$, and there exist some positive integers $\nu_1,\dots,\nu_q$
such that:
\begin{equation*}
\forall \, \xi=(\xi_1,\dots,\xi_d) \in \R^d \setminus \{ 0 \} \, ,\quad
\det \Big[ \tau \, I+\sum_{j=1}^d \xi_j \, B_j \Big] =\prod_{k=1}^q \big( \tau+\lambda_k(\xi) \big)^{\nu_k} \, .
\end{equation*}
Moreover the eigenvalues $\lambda_1(\xi),\dots,\lambda_q(\xi)$ are semi-simple (their algebraic multiplicity
equals their geometric multiplicity) and satisfy $\lambda_1(\xi)<\dots<\lambda_q(\xi)$ for all $\xi \in \R^d
\setminus \{ 0\}$.
\end{assumption}

\noindent For simplicity, we restrict our analysis to noncharacteristic boundaries and therefore make the
following:

\begin{assumption}
\label{assumption2}
The matrix $B_d$ is invertible and the matrix $B:=\psi(0)$ has maximal rank, its rank $p$ being equal to the
number of positive eigenvalues of $B_d$ (counted with their multiplicity). Moreover, the integer $p$ satisfies
$1 \le p \le N-1$.
\end{assumption}

In the normal modes analysis for  \eqref{2}, one first performs a Laplace transform in the time variable $t$ and
a Fourier transform in the tangential space variables $y$.  We let $\tau-i\, \gamma \in \C$ and $\eta \in \R^{d-1}$
denote the dual variables of $t$ and $y$.  We introduce the symbol
\begin{equation*}
{\mathcal A}(\zeta):= -i \, B_d^{-1} \left( (\tau-i\gamma) \, I +\sum_{j=1}^{d-1} \eta_j \, B_j \right)
\, ,\quad \zeta:=(\tau-i\gamma,\eta) \in \C \times \R^{d-1} \, .
\end{equation*}
For future use, we also define the following sets of frequencies:
\begin{align*}
& \Xi := \Big\{ (\tau-i\gamma,\eta) \in \C \times \R^{d-1} \setminus (0,0) : \gamma \ge 0 \Big\} \, ,
& \Sigma := \Big\{ \zeta \in \Xi : \tau^2 +\gamma^2 +|\eta|^2 =1 \Big\} \, ,\\
& \Xi_0 := \Big\{ (\tau,\eta) \in \R \times \R^{d-1} \setminus (0,0) \Big\} = \Xi \cap \{ \gamma = 0 \} \, ,
& \Sigma_0 := \Sigma \cap \Xi_0 \, .
\end{align*}

Two key objects in our analysis are the hyperbolic region and the glancing set that are defined as follows:

\begin{defn}
\label{def1}
\begin{itemize}
 \item The hyperbolic region ${\mathcal H}$ is the set of all $(\tau,\eta) \in \Xi_0$ such that the matrix
       ${\mathcal A}(\tau,\eta)$ is diagonalizable with purely imaginary eigenvalues.

 \item Let ${\bf G}$ denote the set of all $(\tau,\xi) \in \R \times \R^d$ such that $\xi \neq 0$ and there exists
       an integer $k \in \{1,\dots,q\}$ satisfying:
\begin{equation*}
\tau + \lambda_k(\xi) = \dfrac{\partial \lambda_k}{\partial \xi_d} (\xi) = 0 \, .
\end{equation*}
If $\pi ({\bf G})$ denotes the projection of ${\bf G}$ on the $d$ first coordinates (in other words $\pi (\tau,\xi)
=(\tau,\xi_1,\dots,\xi_{d-1})$ for all $(\tau,\xi)$), the glancing set ${\mathcal G}$ is ${\mathcal G} :=
\pi ({\bf G}) \subset \Xi_0$.
\end{itemize}
\end{defn}

\noindent We recall the following result that is due to Kreiss \cite{K} in the strictly hyperbolic case (when all integers
$\nu_j$ in Assumption \ref{assumption1} equal $1$) and to M\'etivier \cite{Met} in our more general framework:

\begin{prop}[\cite{K,Met}]
\label{thm1}
Let Assumptions \ref{assumption1} and \ref{assumption2} be satisfied. Then for all $\zeta \in \Xi \setminus
\Xi_0$, the matrix ${\mathcal A}(\zeta)$ has no purely imaginary eigenvalue and its stable subspace $\E^s
(\zeta)$ has dimension $p$. Furthermore, $\E^s$ defines an analytic vector bundle over $\Xi \setminus \Xi_0$
that can be extended as a continuous vector bundle over $\Xi$.
\end{prop}

\noindent For all $(\tau,\eta) \in \Xi_0$, we let $\E^s(\tau,\eta)$ denote the continuous extension of $\E^s$ to the
point $(\tau,\eta)$.  The analysis in \cite{Met} shows that away from the glancing set ${\mathcal G} \subset \Xi_0$,
$\E^s(\zeta)$ depends analytically on $\zeta$, and the hyperbolic region ${\mathcal H}$ does not contain any
glancing point.

To treat the case when the boundary operator in \eqref{2}(b) is independent of $u$, meaning $\psi(\eps u)
\equiv \psi(0) =:B$, we make the following \emph{weak stability assumption} on the problem $(L(\partial),B)$.

\begin{assumption}
\label{assumption3}
\begin{itemize}
 \item For all $\zeta \in \Xi \setminus \Xi_0$, $\text{\rm Ker} B \cap \E^s (\zeta) = \{ 0\}$.

 \item The set $\Upsilon_0 := \{ \zeta \in \Sigma_0 : \text{\rm Ker} B \cap \E^s (\zeta) \neq \{ 0\} \}$ is
       nonempty and included in the hyperbolic region ${\mathcal H}$.
       
 \item For all $\underline{\zeta} \in \Upsilon_0$, there exists a neighborhood ${\mathcal V}$ of $\underline{\zeta}$
       in $\Sigma$, a real valued ${\mathcal C}^\infty$ function $\sigma$ defined on ${\mathcal V}$, a basis
       $E_1(\zeta),\dots,E_p(\zeta)$ of $\E^s(\zeta)$ that is of class ${\mathcal C}^\infty$ with respect to
       $\zeta \in {\mathcal V}$, and a matrix $P(\zeta) \in \text{\rm GL}_p (\C)$ that is of class
       ${\mathcal C}^\infty$ with respect to $\zeta \in {\mathcal V}$, such that
\begin{equation*}
\forall \, \zeta \in {\mathcal V} \, ,\quad B \, \begin{pmatrix}
E_1(\zeta) & \dots & E_p(\zeta) \end{pmatrix}
= P(\zeta) \, \text{\rm diag } \big( \gamma +i\, \sigma (\zeta),1,\dots,1 \big) \, .
\end{equation*}
\end{itemize}
\end{assumption}

\noindent For comparison and later reference we recall the following definition.

\begin{defn}[\cite{K}]
\label{ustable}
As before let $p$ be the number of positive eigenvalues of $B_d$. The problem $(\tilde L(\partial),B)$ is said to
be \emph{uniformly stable} or to satisfy the \emph{uniform Lopatinskii condition} if
\begin{equation*}
B \, : \, \E^s (\zeta) \longrightarrow \C^p
\end{equation*}
is an isomorphism for all $\zeta\in \Sigma$.
\end{defn}

\begin{rem}
\label{careful}
\textup{Observe that if $(\tilde L(\partial),B)$ satisfies the uniform Lopatinskii condition, continuity implies that this
condition still holds for $(\tilde L(\partial),B+\dot\psi)$, where $\dot\psi$ is any sufficiently small perturbation of $B$.
Hence the uniform Lopatinskii condition is a convenient framework for nonlinear perturbation. The analogous
statement may not be true when $(\tilde L(\partial),B)$ is only weakly stable. Remarkably, weak stability persists
under perturbation in the so-called WR class exhibited in \cite{BRSZ}, and Assumption \ref{assumption3}
above is a convenient equivalent definition of the WR class (see \cite[Appendix B]{CG} or \cite[Chapter 8]{BS}). 
In order to handle general nonlinear boundary conditions as in \eqref{2} we shall strengthen Assumption 
\ref{assumption3} in Assumption \ref{nonlinbc} below.}
\end{rem}

\textbf{Boundary and interior phases.}
We consider a planar real phase $\phi_0$ defined on the boundary:
\begin{equation}
\label{phasebord}
\quad \phi_0(t,y) :=
\underline{\tau} \, t +\underline{\eta} \cdot y \, ,\quad \tauetabar \in \Xi_0 \, .
\end{equation}
As follows from earlier works, see e.g. \cite{MA}, oscillations or pulses on the boundary associated with the 
phase $\phi_0$ give rise to oscillations or pulses in the interior associated with some planar phases $\phi_m$. 
These phases are characteristic for the hyperbolic operator $L(\partial)$ and their trace on the boundary 
$\{ x_d=0\}$ equals $\phi_0$. For now we make the following:

\begin{assumption}
\label{a8}
The phase $\phi_0$ defined by \eqref{phasebord} satisfies $\tauetabar \in \Upsilon_0$. In particular
$\tauetabar \in {\mathcal H}$.
\end{assumption}

\noindent Thanks to Assumption \ref{a8}, we know that the matrix ${\mathcal A} \tauetabar$ is
diagonalizable with purely imaginary eigenvalues. These eigenvalues are denoted $i\, {\omega}_1,
\dots,i\, {\omega}_M$, where the ${\omega}_m$'s are real and pairwise distinct. The
${\omega}_m$'s are the roots (and all the roots are real) of the dispersion relation:
\begin{equation*}
\det \Big[ \underline{\tau} \, I+\sum_{j=1}^{d-1} \underline{\eta}_j \, B_j +\omega \, B_d \Big] = 0 \, .
\end{equation*}
To each root ${\omega}_m$ there corresponds a unique integer $k_m \in \{ 1,\dots,q\}$ such that
$\underline{\tau} + \lambda_{k_m} (\underline{\eta},{\omega}_m)=0$. We can then define the following
real\footnote{If $\tauetabar$ does not belong to the hyperbolic region ${\mathcal H}$, some of the phases
$\varphi_m$ may be complex, see e.g. \cite{H,W2,Le,Mar}. Moreover, glancing phases introduce a new
scale $\sqrt{\eps}$ as well as boundary layers.} phases and their associated group velocities:
\begin{equation}
\label{phases}
\forall \, m =1,\dots,M \, ,\quad \phi_m (x):= \phi_0(t,y)+{\omega}_m \, x_d \, ,\quad
{\bf v}_m := \nabla \lambda_{k_m} (\underline{\eta},{\omega}_m) \, .
\end{equation}
Let us observe that each group velocity ${\bf v}_m$ is either incoming or outgoing with respect to the space
domain $\R^d_+$: the last coordinate of ${\bf v}_m$ is nonzero. This property holds because $\tauetabar$
does not belong to the glancing set ${\mathcal G}$. We can therefore adopt the following classification:

\begin{defn}
\label{def2}
The phase $\phi_m$ is incoming if the group velocity ${\bf v}_m$ is incoming (that is, when $\partial_{\xi_d}
\lambda_{k_m} (\underline{\eta},{\omega}_m)>0$), and it is outgoing if the group velocity ${\bf v}_m$
is outgoing ($\partial_{\xi_d} \lambda_{k_m} (\underline{\eta},{\omega}_m) <0$).
\end{defn}

\noindent In all that follows, we let ${\mathcal I}$ denote the set of indices $m \in \{ 1,\dots,M\}$ such
that $\phi_m$ is an incoming phase, and ${\mathcal{O}}$ denote the set of indices $m \in \{ 1,\dots,M\}$ such
that $\phi_m$ is an outgoing phase.   If $p\geq 1$, then $\cI$ is nonempty, while if $p\leq N-1$, $\cO$ is
nonempty (see Lemma \ref{lem1} below).  We  will use the notation:
\begin{align*}
&L(\tau,\xi) := \tau \, I +\sum_{j=1}^d \xi_j \, B_j \, ,\; \;
\tilde L(\beta,\omega_m):=\omega_m \, I +\sum_{k=0}^{d-1} \beta_k \, A_k,\\
&\;\beta:=(\utau,\ueta), \;\;x'=(t,y),\;\; \phi_0(x')=\beta\cdot x'.
\end{align*}
For each phase $\phi_m$, ${\rm d}\phi_m$ denotes the differential of the function $\phi_m$ with respect to its
argument $x=(t,y,x_d)$. It follows from Assumption \ref{assumption1} that the eigenspace of ${\mathcal A}
(\beta)$ associated with the eigenvalue $i\, {\omega}_m$ coincides with the kernel of $L({\rm d}\phi_m)$ and has 
dimension $\nu_{k_m}$. The following result is a direct consequence of the diagonalizability of ${\mathcal A} 
(\beta)$, see \cite{CG} for the proof:

\begin{lem}
\label{lem2}
The space $\C^N$ admits the decomposition:
\begin{equation}
\label{decomposition2}
\C^N = \oplus_{m=1}^M \, \ker L({\rm d}\phi_m)
\end{equation}
and each vector space in \eqref{decomposition2} admits a basis of real vectors. If we let $P_1,\dots,P_M$ denote
the projectors associated with the decomposition \eqref{decomposition2}, then there holds $\text{\rm Im } \tilde L 
({\rm d}\phi_m) = \ker P_m$ for all $m=1,\dots,M$.
\end{lem}

The following well-known lemma, whose proof is also recalled in \cite{CG}, gives a useful decomposition of $\E^s$ 
in the hyperbolic region.

\begin{lem}
\label{lem1}
The stable subspace $\E^s(\beta)$ admits the decomposition:
\begin{equation}
\label{decomposition1}
\E^s (\beta) = \oplus_{m \in {\mathcal I}} \, \ker L({\rm d} \phi_m) \, ,
\end{equation}
and each vector space in the decomposition \eqref{decomposition1} admits a basis of real vectors.
\end{lem}

By Assumption \ref{assumption3} we know that the vector space $\ker B \, \cap \, \E^s(\beta)$ is one-dimensional; 
moreover, it admits a real basis because $B$ has real coefficients and $\E^s(\beta)$ has a real basis. This vector 
space is therefore spanned by a vector $e \in \R^N \setminus \{ 0\}$ that we can decompose in a unique way by 
using Lemma \ref{lem1}:
\begin{equation}
\label{defe}
\ker B \, \cap \, \E^s(\beta) = \text{\rm Span } \{e\} \, ,\quad 
e = \sum_{m \in {\mathcal I}} e_m \, ,\quad P_m \, e_m =e_m \, .
\end{equation}
Each vector $e_m$ in \eqref{defe} has real components. We also know that the vector space $B \, \E^s(\beta)$ 
is $(p-1)$-dimensional. We can therefore write it as the kernel of a real linear form:
\begin{equation}
\label{defb}
B \, \E^s(\beta) = \left\{ X \in \C^p \, ,\quad b \cdot X =\sum_{j=1}^p b_j \, X_j = 0 \right\} \, ,
\end{equation}
for a suitable vector $b \in \R^p \setminus \{ 0\}$.
\bigskip

To formulate our last assumption we observe first that for every point $\uzeta\in\cH$ there is a neighborhood
$\cV$ of $\uzeta$ in $\Sigma$ and a $C^\infty$ conjugator $Q_0(\zeta)$ defined on $\cV$ such that
\begin{align}\label{conjugator}
Q_0(\zeta)\cA(\zeta)Q_0^{-1}(\zeta)=\begin{pmatrix}
i\omega_1(\zeta) I_{n_1} & \; & 0\\
\; & \ddots & \;\\
0 & \; & i\omega_J(\zeta) I_{n_J} \end{pmatrix} =:\bD_1(\zeta),
\end{align}
where the $\omega_j$ are real when $\gamma=0$ and there is a constant $c>0$ such that either
\begin{equation*}
\re (i\omega_j) \leq -c\, \gamma \text{ or } \re (i\omega_j) \geq c\, \gamma \text{ for all } \zeta \in \cV.
\end{equation*}
In view of Lemma \ref{lem1} we can choose the first $p$ columns of $Q^{-1}_0(\zeta)$ to be a basis of 
$\E^s(\zeta)$, and we write
\begin{equation*}
Q_0^{-1}(\zeta)=[Q_{in}(\zeta) \;Q_{out}(\zeta)].
\end{equation*}
Choose $J'$ so that the first $J'$ blocks of $\bD_1$ lie in the first $p$ columns, and the remaining blocks 
in the remaining $N-p$ columns. Thus, $\re (i\omega_j) \le -c\, \gamma$ for $1\leq j\leq J'$.

Observing that the linearization of the boundary condition in \eqref{2} is
\begin{equation*}
\dot u \longmapsto \psi(\eps u) \dot u +[{\rm d}\psi(\eps u) \, \dot u] \, \eps u \, ,
\end{equation*}
we define the operator
\begin{align}\label{ca2}
\cB(v_1,v_2) \, \dot u := \psi(v_1) \dot u +[{\rm d}\psi(v_1) \, \dot u] \, v_2 \, ,
\end{align}
which appears in Assumption \ref{nonlinbc} below. For later use we also define
\begin{align}\label{caa}
\cD(v_1,v_2) \, \dot u := D(v_1) \dot u +[{\rm d}D(v_1) \, \dot u] \, v_2 \, ,
\end{align}
as well as
\begin{align}\label{caaa}
\cB(v_1) :=\cB(v_1,v_1) \, , \; \; \cD(v_1) := \cD(v_1,v_1) \, .
\end{align}

We now state the weak stability assumption that we make when considering the general case of nonlinear
boundary conditions in \eqref{2}.

\begin{assumption}
\label{nonlinbc}
\begin{itemize}
 \item There exists a neighborhood $\cO$ of $(0,0)\in\R^{2N}$ such that for all $(v_1,v_2)\in\cO$ and all
          $\zeta \in \Xi \setminus \Xi_0$, $\ker \;\cB(v_1,v_2) \cap \E^s (\zeta) = \{ 0\}$. For each $(v_1,v_2) \in
          \cO$ the set $\Upsilon(v_1,v_2):=\{\zeta\in\Sigma_0:\ker \; \cB(v_1,v_2) \cap \E^s (\zeta) \neq \{ 0\} \}$
          is nonempty and included in the hyperbolic region ${\mathcal H}$. Moreover, if we set $\Upsilon :=
          \cup_{(v_1,v_2)\in\cO}\Upsilon(v_1,v_2)$, then $\overline{\Upsilon}\subset \cH$ (closure in $\Sigma_0$).

 \item For every $\uzeta\in\overline{\Upsilon}$ there exists a neighborhood $\cV$ of $\uzeta$ in $\Sigma$
          and a $C^\infty$ function $\sigma(v_1,v_2,\zeta)$ on $\cO\times\cV$ such that for all $(v_1,v_2,\zeta)
          \in \cO\times \cV$ we have $\text{\rm Ker}\; \cB(v_1,v_2) \cap \E^s (\zeta) \neq \{ 0\}$ if and only if
          $\zeta\in\Sigma_0$ and $\sigma(v_1,v_2,\zeta)=0$.

          \qquad Moreover, there exist matrices $P_i(v_1,v_2,\zeta)\in\text{\rm GL}_p (\C)$, $i=1,2,$ of class
          $C^\infty$ on $\cO\times\cV$ such that $\forall \, (v_1,v_2,\zeta) \in \cO\times {\mathcal V}$
\begin{equation}\label{ca3a}
P_1(v_1,v_2,\zeta)\cB(v_1,v_2)Q_{in}(\zeta)P_2(v_1,v_2,\zeta) =\text{\rm diag }
\big( \gamma +i\, \sigma (v_1,v_2,\zeta),1,\dots,1 \big) \, .
\end{equation}
 \end{itemize}
\end{assumption}

For nonlinear boundary conditions, the phase $\phi_0$ in \eqref{phasebord} is assumed to satisfy
$\tauetabar \in \Upsilon(0,0)$, or in other words the intersection $\ker B \cap \E^s \tauetabar$ is not
reduced to $\{ 0\}$. The set $\Upsilon_0$ in Assumption \ref{assumption3} is a short notation for
$\Upsilon (0,0)$. The phases $\phi_m$ are still defined by \eqref{phases} and thus only depend
on $L(\partial)$ and $B$ and not on the nonlinear perturbations $f_0$ and $\psi (\eps \, u) -\psi(0)$
added in \eqref{2}.

\begin{rem}\label{ca3}
1) \textup{The properties stated in Assumption \ref{nonlinbc} are just a convenient description of the requirements
for belonging to the WR class of \cite{BRSZ}. Like the uniform Lopatinskii condition, Assumption \ref{nonlinbc} can
in practice be verified by hand via a ``constant-coefficient" computation. More precisely, for $(v_1,v_2)$ near
$(0,0)\in\R^{2N}$ and $\zeta\in\Sigma$, one can define (see, e.g., \cite[chapter 4]{BS}) a Lopatinskii determinant
$\Delta(v_1,v_2,\zeta)$ that is $C^\infty$ in $(v_1,v_2)$, analytic in $\zeta=(\tau-i\gamma,\eta)$ on $\Sigma
\setminus \cG$, continuous on $\Sigma$, and satisfies
\begin{equation*}
\Delta(v_1,v_2,\zeta)=0\text{ if and only if }\text{\rm Ker}\; \cB(v_1,v_2) \cap \E^s (\zeta) \neq \{ 0\}.
\end{equation*}
In particular, $\Delta(v_1,v_2,\cdot)$ is real-analytic on $\cH$.}

\textup{It is shown in Remark 1.13 of \cite{CGW3} that Assumption \ref{nonlinbc} holds provided
\begin{equation}\label{ca4}
\emptyset\neq\{\zeta\in\Sigma: \Delta(0,0,\zeta)=0\}\subset \cH\text{ and }
\Delta(0,0,\uzeta)=0 \Rightarrow \partial_\tau \Delta(0,0,\uzeta)\neq 0,
\end{equation}
and it thus only involves a weak stability property for the linearized problem at $(v_1,v_2)=(0,0)$. 
Instead of assuming \eqref{ca4}, we have stated Assumption \ref{nonlinbc} in a form that is more directly
applicable to the error analysis of Theorem \ref{main2}.}

2) \textup{To prove the basic estimate for the linearized singular system, Proposition \ref{i5z}, and to construct 
the exact solution $U_\eps$ to the singular system \eqref{15p} below, it is enough to require that the analogue of 
Assumption \ref{nonlinbc} holds when $\cB(v_1,v_2)$ is replaced by $\cB(v_1)$ given in \eqref{caaa}. However, 
for the error analysis of section \ref{error} in the case of nonlinear boundary conditions, we need Assumption 
\ref{nonlinbc} as stated.}
\end{rem}

\subsection{Main results}
\label{mr}

For each $m\in\{1,\dots,M\}$ we let
\begin{equation*}
r_{m,k}, \;k=1,\dots,\nu_{k_m}
\end{equation*}
denote a basis of $\ker L({\rm d}\phi_m)$ consisting of real vectors. In section \ref{error} we shall construct a
``corrected" approximate solution $u^c_\eps$ of \eqref{2} of the form
\begin{equation}
\label{a10}
u^c_\eps(x)= \Big[ \cU^0 ( x,\theta_0,\xi_d) +\eps \, \cU^1 ( x,\theta_0,\xi_d)
+\eps^2 \, \cU^2_p ( x,\theta_0,\xi_d ) \Big]|_{\theta_0=\frac{\phi_0}{\eps},\xi_d=\frac{x_d}{\eps}},
\end{equation}
where
\begin{align}
\label{a11}
\begin{split}
&\cU^0 \left( x,\theta_0,\xi_d \right) =\sum_{m\in\cI} \, \sum^{\nu_{k_m}}_{k=1}
\sigma_{m,k}\left(x,\theta_0+\omega_m\xi_d\right) \, r_{m,k} \, ,\\
&\cU^1 \left( x,\theta_0,\xi_d\right) =\sum_{m\in \cI} \, \sum^{\nu_{k_m}}_{k=1}
\tau_{m,k}\left(x,\theta_0+\omega_m\xi_d\right) \, r_{m,k} -  {\bf R}_\infty \left(\tilde L(\partial)\cU^0+D(0)\cU^0\right),
\end{split}
\end{align}
and ${\bf R}_\infty$ is the integral operator defined in \eqref{R}. The $\sigma_{m,k}$'s and $\tau_{m,k}$'s 
are scalar functions decaying in $\theta_m:=\theta_0+\omega_m\xi_d$. The $\sigma_{m,k}$'s satisfy 
\begin{equation*}
\int_\R \sigma_{m,k}(x,\theta) \, {\rm d}\theta =0 \, ,\quad  \text{\rm (moment zero)}
\end{equation*}
and  describe the propagation of oscillations with phase $\phi_m$ and group velocity ${\bf v}_m$. 
The second corrector $\eps^2 \, \cU^2_p(x,\theta_0,\xi_d)$ in \eqref{a10} is a more complex object 
constructed in the error analysis of section \ref{error}.

The following theorem, our main result, is an immediate corollary of the more precise Theorem \ref{main2}. 
Here we let $\Omega_T :=\{(x,\theta_0)=(t,y,x_d,\theta_0)\in\R^{d+1}\times \R^1: x_d\geq 0, t <T\}$, and 
$b\Omega_T :=\{(t,y,\theta_0) \in \R^d \times \R^1: t <T\}$. The spaces $E^s_T$ appearing in the theorem are
\begin{equation}
\label{esp}
E^s_T :=C(\R_{x_d},H^s(b\Omega_T))\cap L^2(\R_{x_d},H^{s+1}(b\Omega_T)),
\end{equation}
while the Sobolev spaces with $\theta$-weights $\Gamma^s_T$ and $b\Gamma^s_T$ are defined in Definition 
\ref{c1}. With $[r]$ denoting the smallest integer greater than $r$, we define
\begin{equation}
\label{a11z}
a_0:=\left[\frac{d+1}{2}\right],\; a_1:=\left[\frac{d+1}{2}\right]+M_0,\; a :=\max(2a_0+3,a_1+1), \; \tilde a :=2a-a_0 \, ,
\end{equation}
where $M_0:=3\, d+5$ is fixed according to the regularity requirements in the singular pseudodifferential 
calculus (see Appendix \ref{calculus}).

\begin{theo}
\label{main}
We make Assumptions \ref{assumption1}, \ref{assumption2}, \ref{assumption3}, and \ref{a8} when the boundary 
condition in \eqref{2} is linear ($\psi(\eps u) \equiv \psi(0)$); in the general case we substitute Assumption 
\ref{nonlinbc} for Assumption \ref{assumption3}. Fix $T>0$, set $M_0:=3\, d+5$, and let $a$ and $\tilde a$ 
be as in \eqref{a11z}.

Consider the semilinear boundary problem \eqref{2}, where $G\in b\Gamma^{\tilde a}_T$. There exists $\eps_0>0$ 
such that if $\langle G \rangle_{b\Gamma^{a+2}_T}$ is small enough, there exists a unique function $U_\eps 
(x,\theta_0)\in E^{a-1}_T$ satisfying the singular system \eqref{15p} on $\Omega_T$ such that
\begin{equation*}
u_\eps(x):=U_\eps\left(x,\frac{\phi_0}{\eps}\right)
\end{equation*}
is an exact solution of \eqref{2} on $(-\infty,T] \times \overline{\R}^d_+$ for $0<\eps \leq \eps_0$. In addition 
there exists a profile $\cU^0(x,\theta_0,\xi_d)$ as in \eqref{a11}, whose components $\sigma_{m,k}$ lie in 
$\Gamma^{a+1}_T$, such that the approximate solution defined by
\begin{equation*}
u^{app}_\eps :=\cU^0 \left(x,\frac{\phi_0}{\eps},\frac{x_d}{\eps} \right)
\end{equation*}
satisfies
\begin{equation*}
\lim_{\eps\to 0}|u_\eps-u^{app}_\eps|_{L^\infty}=0\quad \text{\rm on }(-\infty,T]\times \overline{\R}^d_+.
\end{equation*}
In fact we have the rate of convergence
\begin{equation*}
|u_\eps-u^{app}_\eps|_{L^\infty}\leq C\, \eps^{\frac{1}{2M_1+3}} \, ,\quad 
\text{\rm where } M_1:=\left[ \frac{d}{2}+3\right] \, .
\end{equation*}
\end{theo}

Observe that although the boundary data in the problem \eqref{2} is of size $O(\eps)$, the approximate solution 
$u^{app}_\eps$ is of size $O(1)$, exhibiting an amplification due to the weak stability at frequency $\beta$. The 
main information provided by Theorem \ref{main} is that this amplification does not rule out existence of a smooth 
solution on a fixed time interval, that is it does not trigger a violent instability, at least in this weakly nonlinear regime. 
Although the first and second correctors $\cU^1$ and $\cU^2_p$ do not appear in the statement of Theorem 
\ref{main}, they are essential in the error analysis. As an application of Theorem \ref{main}, an easy analysis of 
the equations for the leading profile $\cU^0$ (see \eqref{m6} or \eqref{r8}) gives the qualitative information about 
the exact solution described below in Remark \ref{13a}(a).

\begin{rem}\label{13a}
\textup{a)\; In order to avoid some technicalities we have stated our main result for the problem \eqref{2} where all 
data is $0$ in $t<0$. This result easily implies a similar result in which outgoing pulses defined in $t<0$ of amplitude 
$O(\eps)$ and wavelength $\eps$ give rise to reflected pulses of amplitude $O(1)$. In either formulation, inspection 
of the profile equations  shows that the pulses of amplitude $O(1)$ emanate from a radiating pulse that propagates 
in the boundary along a characteristic of the Lopatinskii determinant.}

\textup{b)\; We have decided to fix $T>0$ at the start and choose data small enough so that a solution to the nonlinear 
problem exists up to time $T$.  With only minor changes one can fix the data  in the problem ($G$ in \eqref{2}) at the 
start, and then choose $T$ small enough so that a solution to the nonlinear problem exists up to time $T$.}
\end{rem}

In the remainder of this introduction, we discuss the construction of exact solutions, the construction of the 
approximate solution $\cU^0$, and the error analysis. The proofs are given in Sections \ref{profanalysis}, 
\ref{exact}, and \ref{error}. We hope that the analysis developed in this article will be useful in justifying
quasilinear amplification phenomena such as the Mach stems or kink modes formation of \cite{MR,AM}.

\subsection{Exact solutions and singular systems}

\emph{\quad} The theory of weakly stable hyperbolic initial boundary value problems fails to provide a solution 
of the system \eqref{2} that exists on a fixed time interval independent of $\eps$\footnote{This would be true 
even for problems $(\tilde L(\partial),B)$ that are uniformly stable in the sense of Definition \ref{ustable}.}. In 
order to obtain such an exact solution to the system \eqref{2} we adopt the strategy of studying an associated 
singular problem first used by \cite{JMR} for an initial value problem in free space. We look for a solution of the form
\begin{align}\label{14p}
u_\eps(x)=U_\eps(x,\theta_0)|_{\theta_0=\frac{\phi_0(x')}{\eps}},
\end{align}
where $U_\eps(x,\theta_0)$ satisfies the singular system derived by substituting \eqref{14p} into the problem 
\eqref{2}. Recalling that $\tilde L(\partial)=\partial_d +\sum^{d-1}_{j=0}A_j \, \partial_j$, we obtain:
\begin{align}\label{15p}
\begin{split}
&\partial_d U_\eps+\sum^{d-1}_{j=0}A_j\left(\partial_j +\frac{\beta_j \partial_{\theta_0}}{\eps}\right) U_\eps
+D(\eps U_\eps)\, U_\eps=:\\
&\qquad\qquad\partial_d U_\eps+\mathbb{A}\left(\partial_{x'}+\frac{\beta \partial_{\theta_0}}{\eps}\right)U_\eps
+D(\eps U_\eps)\, U_\eps=0\\
&\psi(\eps U_\eps) \, U_\eps|_{x_d=0} =\eps \, G(x',\theta_0),\\
&U_\eps =0 \text{ in } t<0.
\end{split}
\end{align}

In the wavetrain case the function $U_\eps(x,\theta_0)$ is taken to be periodic in $\theta_0$, but in the case 
of pulses the variable $\theta_0\in\R$ is treated as just another tangential variable on an equal footing with $x'$. 
The special difficulties presented by such singular problems when there is a boundary are described in detail in 
the introductions to \cite{W1} and \cite{CGW1}; these difficulties are present for both wavetrains and pulses. In 
particular we mention: (a) symmetry assumptions on the matrices $B_j$ appearing in the problem \eqref{0} 
equivalent to \eqref{2} are generally of no help in obtaining an $L^2$ estimate for \eqref{15p} (boundary 
conditions satisfying Assumption \ref{assumption3} can not be maximally dissipative, see \cite{CG}); (b) one 
cannot control $L^\infty$ norms just by estimating tangential derivatives $\partial^\alpha_{(x',\theta_0)}U_\eps$ 
because \eqref{15p} is not a hyperbolic problem in the $x_d$ direction\footnote{For initial value problems in free 
space, one \emph{can} control $L^\infty$ norms just by estimating enough derivatives tangent to time slices $t=c$.}; 
moreover, even if one has estimates of tangential derivatives  uniform with respect to $\eps$, because of the 
factors $1/\eps$ in \eqref{15p} one cannot just use the equation to control $\partial_d U_\eps$ and thereby 
control $L^\infty$ norms.

In \cite{CGW2} a class of singular pseudodifferential operators, acting on functions $U(x',\theta_0)$
and having the form
\begin{equation}
\label{16p}
p_D U(x',\theta_0)=\frac{1}{(2\, \pi)^{d+1}} \, \int_{\R^d \times \R} {\rm e}^{ix' \cdot \xi' +i\theta_0 k} \,
p\left(\eps V(x',\theta_0),\xi'+\frac{k\beta}{\eps},\gamma\right) \, \widehat{U}(\xi',k) \, {\rm d}\xi' 
\, {\rm d}k, \; \; \gamma\geq 1,
\end{equation}
was introduced to deal with these difficulties. Observe that the differential operator $\mathbb{A}$ appearing in 
\eqref{15p} can be expressed in this form. Operators like \eqref{16p}, but with the integral over $k$ replaced 
by a sum over $k\in\Z$, were first used in \cite{W1} to prove uniform estimates for quasilinear singular systems 
arising in the uniformly stable wavetrain case. The calculi of \cite{CGW2} are inspired by the calculus of \cite{W1}, 
but incorporate improvements needed for the study of amplification in both the wavetrain and pulse cases.

Under the weak stability assumptions of the present paper, Assumption \ref{assumption3} or Assumption 
\ref{nonlinbc}, the basic $L^2$ estimate for the problem $(\tilde L(\partial),B)$ and for its singular analogue 
both exhibit loss of one tangential derivative. To derive energy estimates with a loss of derivative, one needs 
the remainder in the composition of two zero-order singular operators to be a smoothing operator, and not 
just bounded on $L^2$ with small norm as in \cite{W1}. The calculi constructed in \cite{CGW2} have this 
smoothing property. The improved wavetrain calculus was used in \cite{CGW3} to study amplification of 
wavetrains. As noted in that paper, the arguments used to prove the main $L^2$ estimate for the linearized 
singular problem and the higher order tame estimates carry over {\it verbatim} to the pulse case. The 
Nash-Moser iteration used in \cite{CGW3} to construct an exact solution $U_\eps(x,\theta_0)$ of \eqref{15p} 
on $\Omega_T$ also carries over verbatim to the pulse case. In section \ref{exact} we state the main estimates 
for the linearized singular problem in the pulse case along with the existence theorem for the nonlinear singular 
system \eqref{15p}, and refer to \cite{CGW3} for the proofs. The tame estimate for the linearized singular 
problem is given in terms of $E^s_T$ spaces (recall \eqref{esp}), and these spaces are used as well in the 
Nash-Moser iteration. The singular pulse calculus summarized in Appendix \ref{calculus} must, of course, 
be substituted for the wavetrain calculus used in \cite{CGW3}.

\subsection{Derivation of the profile equations}
\label{derivation}

\emph{\quad} At first we work formally, looking for a corrected approximate solution of the form
\begin{equation*}
u^c_\eps(x)=\Big[ \cU^0(x,\theta_0,\xi_d)+\eps \, \cU^1(x,\theta_0,\xi_d)+\eps^2 \, \cU^2(x,\theta_0,\xi_d) 
\Big]|_{\theta_0=\frac{\phi_0}{\eps},\xi_d=\frac{x_d}{\eps}}.
\end{equation*}
Plugging into the system \eqref{2} and setting coefficients of successive powers of $\eps$ equal to zero,
we obtain interior profile equations
\begin{align}\label{13}
\begin{split}
\eps^{-1} &: \quad \tilde\cL(\partial_{\theta_0},\partial_{\xi_d})\cU^0=0\\
\eps^0 &: \quad \tilde \cL(\partial_{\theta_0},\partial_{\xi_d})\cU^1+\tilde L(\partial)\cU^0+D(0)\cU^0=0\\
\eps^1 &: \quad \tilde \cL(\partial_{\theta_0},\partial_{\xi_d})\cU^2+\tilde L(\partial)\cU^1+D(0)\cU^1 
+({\rm d}D(0) \, \cU^0) \cU^0=0,
\end{split}
\end{align}
where
\begin{equation*}
\tilde\cL(\partial_{\theta_0},\partial_{\xi_d}) 
:=\partial_{\xi_d}+\sum^{d-1}_{j=0}\beta_j A_j\partial_{\theta_0}=\partial_{\xi_d}+i\cA(\beta)\partial_{\theta_0} .
\end{equation*}
Similarly, we obtain the  boundary equations on $\{x_d=0,\xi_d=0\}$,
\begin{align}\label{14}
\begin{split}
&\eps^0:B\, \cU^0=0\\
&\eps^1:B\, \cU^1+({\rm d}\psi(0) \, \cU^0) \cU^0 =G(x',\theta_0).
\end{split}
\end{align}

For ease of exposition we will begin by studying these equations in the $3\times 3$ strictly hyperbolic case, 
which contains all the main difficulties. In section \ref{extension} we describe the (minor) changes needed to 
treat the general case. We define the boundary phase $\phi_0:=\beta\cdot x'$, eigenvalues $\omega_j$, and 
interior phases $\phi_j:=\phi_0+\omega_j x_d$ as in Section \ref{assumptions}, where $\beta\in \cH$. We 
suppose $\omega_1$ and $\omega_3$ are incoming (or causal) and $\omega_2$ is outgoing. The 
corresponding right and left eigenvectors of $i\cA(\beta)$ are $r_j$ and $l_j$, $j=1,2,3$, normalized so 
that $l_j \cdot r_k =\delta_{jk}$.

A general function $W(x,\theta_0,\xi_d)$ valued in $\R^3$ can be written
\begin{equation*}
W=W_1(x,\theta_0,\xi_d)r_1+W_2(x,\theta_0,\xi_d)r_2+W_2(x,\theta_0,\xi_d)r_3
\end{equation*}
where we refer to $W_i$, $i=1,3$ as the \emph{incoming} components and $W_2$ as the \emph{outgoing} 
component of $W$. With $F(x,\theta_0,\xi_d)=\sum^3_{i=1}F_i(x,\theta_0,\xi_d)r_i$ valued in $\R^3$ consider 
now an equation of the form
\begin{equation*}
\tilde \cL(\partial_{\theta_0},\partial_{\xi_d})W=F.
\end{equation*}
Using the matrix $\begin{bmatrix}r_1&r_2&r_3\end{bmatrix}$ to diagonalize $i\cA(\beta)$,  we find that the 
$W_i$ must satisfy
\begin{align}\label{21}
(\partial_{\xi_d}-\omega_i\partial_{\theta_0})W_i(x,\theta_0,\xi_d)=l_i\cdot F:=F_i(x,\theta_0,\xi_d),\;i=1,2,3.
\end{align}
The general solution to \eqref{21} is
\begin{align}\label{22}
W_i(x,\theta_0,\xi_d)=\tau^*_i(x,\theta_0+\omega_i\xi_d)+\int^{\xi_d}_0 F_i(x,\theta_0+\omega_i(\xi_d-s),s)\, 
{\rm d}s,
\end{align}
where $\tau^*_i$ is arbitrary. This can be rewritten
\begin{align}\label{23}
\begin{split}
W_i(x,\theta_0,\xi_d)&= \tau^*_i(x,\theta_0+\omega_i\xi_d)+\int_{\R^+} F_i(x,\theta_0+\omega_i(\xi_d-s),s)\, {\rm d}s
-\int_{\xi_d}^{+\infty} F_i(x,\theta_0+\omega_j(\xi_d-s),s)\, {\rm d}s\\
&:=\tau_i(x,\theta_0+\omega_i\xi_d) -\int_{\xi_d}^{+\infty} F_i(x,\theta_0+\omega_i(\xi_d-s),s)\, {\rm d}s \, ,
\end{split}
\end{align}
provided the integrals in \eqref{23} exist.

\begin{defn}[Functions of type $\cF$]\label{30}
Suppose
\begin{align}\label{30aa}
F(x,\theta_0,\xi_d)=\sum_{i=1}^3 F_i(x,\theta_0,\xi_d)r_i,
\end{align}
where each $F_i$ has the form
\begin{align}\label{30a}
\begin{split}
&F_i(x,\theta_0,\xi_d)=\sum_{k=1}^3 f^i_k(x,\theta_0+\omega_k \, \xi_d) 
+\sum_{l \le m=1}^3 g^i_{l,m}(x,\theta_0+\omega_l \, \xi_d) \, h^i_{l,m}(x,\theta_0+\omega_m \, \xi_d),\\
\end{split}
\end{align}
where the functions  $f^i_k(x,\theta)$, $g^i_{l,m}(x,\theta)$, $h^i_{l,m}(x,\theta)$ are real-valued, $C^1$, and decay 
along with their first order partials at the rate $O(\langle\theta\rangle^{-2})$ uniformly with respect to $x$. We then 
say that $F$ is of \emph{type $\cF$}. For such functions $F$, define
\begin{align*}
{\bf E} F(x,\theta_0,\xi_d):=\sum_{j=1}^3 \left( \lim_{T\to\infty} \, \dfrac{1}{T} \, \int^T_0 
l_j\cdot F(x,\theta_0+\omega_j \, (\xi_d-s),s) \, {\rm d}s \right) \, r_j \, .
\end{align*}
\end{defn}

\begin{rem}\label{30ab}
\textup{For $F$ as in \eqref{30a}, we have
\begin{equation*}
{\bf E} F=\sum^3_{i=1} \tilde F_i(x,\theta_0,\xi_d) \, r_i,\text{ where } 
\tilde F_i :=f^i_i(x,\theta_0+\omega_i\xi_d)+g^i_{i,i}(x,\theta_0+\omega_i\xi_d)h^i_{i,i}(x,\theta_0+\omega_i\xi_d).
\end{equation*}}
\end{rem}

\begin{rem}\label{32}
\textup{The definition of ${\bf E}$ can be extended to more general functions. For example, if
\begin{equation*}
F=\sum^3_{i=1}F_i(x,\theta_0+\omega_i\xi_d)\, r_i,
\end{equation*}
where the $F_i(x,\theta)$ are arbitrary continuous functions (not necessarily decaying with respect to $\theta$), 
the limits that define ${\bf E} F$ exist and we have ${\bf E} F=F$.}

\textup{For another example, suppose $F$ is of type $\cF$ and satisfies ${\bf E} F=0$. Define
\begin{align}\label{R}
{\bf R}_\infty F(x,\theta_0,\xi_d):=-\sum_{i=1}^3\left( \int_{\xi_d}^{+\infty} F_i(x,\theta_0+\omega_i(\xi_d-s),s) \, 
{\rm d}s\right)r_i \, .
\end{align}
Then the limits defining ${\bf R}_\infty F$ and ${\bf E} {\bf R}_\infty F$ exist and we have ${\bf E} {\bf R}_\infty F=0$.}
\end{rem}

The following two results are proved in \cite{CW4}.

\begin{prop}[\cite{CW4}]
\label{35}
Suppose $F$ is of type $\cF$ and ${\bf E}F=0$. Then ${\bf R}_\infty F$ is bounded and
\begin{align}\label{35z}
\tilde \cL(\partial_{\theta_0},\partial_{\xi_d}){\bf R}_\infty F ={\bf R}_\infty\tilde \cL(\partial_{\theta_0},\partial_{\xi_d})F 
=F =(I-{\bf E})F.
\end{align}
\end{prop}



\begin{prop}[\cite{CW4}]
\label{31}
Let $F(x,\theta_0,\xi_d)$ be a function of type $\cF$. 
(a)\; Then the equation $\tilde \cL(\partial_{\theta_0},\partial_{\xi_d})\cU =F$ has a solution bounded in $(\theta_0,\xi_d)$ 
if and only if ${\bf E}F=0$.

(b)\; When ${\bf E}F=0$, every $C^1$ solution bounded in $(\theta_0,\xi_d)$  has the form
\begin{equation*}
\cU=\sum_{i=1}^3 \tau_i(x,\theta_0+\omega_i\xi_d) \, r_i +{\bf R}_\infty F 
\text{ with } \tau_i(x,\theta)\in C^1 \text{ and bounded.}
\end{equation*}
Here ${\bf E}\cU=\sum_{i=1}^3 \tau_i(x,\theta_0+\omega_i\xi_d)\, r_i$ and $(I-{\bf E})\cU={\bf R}_\infty F$.

(c)If $\cU$ is of type $\cF$ then
\begin{align}\label{35y}
{\bf E}\tilde \cL(\partial_{\theta_0},\partial_{\xi_d})\cU =\tilde \cL(\partial_{\theta_0},\partial_{\xi_d}){\bf E}\cU=0.
\end{align}
\end{prop}


By applying the operators ${\bf E}$ and ${\bf R}_\infty$ to the equations \eqref{13} formally (for now) using 
the properties \eqref{35z},\eqref{35y}, we obtain:
\begin{align}\label{14a}
\begin{split}
&(a)\; {\bf E} \cU^0=\cU^0,\\
&(b)\; {\bf E} (\tilde L(\partial)\cU^0+D(0)\cU^0)=0,\\
&(c)\; B\cU^0=0\text{ on }x_d=0,\xi_d=0,\\
&(d)\; \cU^0=0\text{ in } t<0
\end{split}
\end{align}
and
\begin{align}\label{14b}
\begin{split}
&(a)\; (I-{\bf E})\cU^1+{\bf R}_\infty(\tilde L(\partial)\cU^0+D(0)\cU^0)=0,\\
&(b)\; {\bf E}\left( \tilde L(\partial)\cU^1+D(0)\cU^1+({\rm d}D(0) \, \cU^0) \cU^0 \right)=0,\\
&(c)\; B {\bf E} \cU^1=G -({\rm d}\psi(0) \, \cU^0) \cU^0 -B(I-{\bf E}) \cU^1 \text{ on }x_d=0, \xi_d=0,\\
&(d)\; \cU^1=0\text{ in }t<0.
\end{split}
\end{align}

To construct $\cU^1$ we write $\cU^1={\bf E}\cU^1+(I-{\bf E})\cU^1$ where
\begin{align}\label{15}
(I-{\bf E})\cU^1=-{\bf R}_\infty(\tilde L(\partial)\cU^0+D(0)\cU^0)={\bf R}_\infty F^0 \, ,
\end{align}
and ${\bf E}\cU^1$ satisfies
\begin{align}\label{15bb}
\begin{split}
&(a)\; {\bf E} (\tilde L(\partial)+D(0)) {\bf E} \cU^1 =-{\bf E} (({\rm d}D(0) \, \cU^0) \cU^0) 
-{\bf E} (\tilde L(\partial) +D(0)) {\bf R}_\infty F^0 \, ,\\
&(b)\; B {\bf E} \cU^1=G -({\rm d}\psi(0) \, \cU^0) \cU^0 -{\bf R}_\infty F^0 \text{ on }x_d=0, \xi_d=0\\
&(c)\; {\bf E} \cU^1=0\text{ in }t<0.
\end{split}
\end{align}

\subsection{Error analysis}
\label{erran}

\emph{\quad} The error analysis is done by studying the singular problem \eqref{m22} satisfied by the difference
\begin{equation*}
W_\eps(x,\theta_0)=U_\eps(x,\theta_0)-\cU_\eps(x,\theta_0),
\end{equation*}
where $U_\eps$ is the exact solution of the semilinear singular problem \eqref{15p} and $\cU_\eps$ is a corrected 
approximate solution of the form
\begin{equation*}
\cU_\eps(x,\theta_0)=\cU^0 (x,\theta_0,\frac{x_d}{\eps}) +\eps \, \cU^1 ( x,\theta_0,\frac{x_d}{\eps})
+\eps^2 \, \cU^2_p ( x,\theta_0,\frac{x_d}{\eps} ).
\end{equation*}
Here $\cU^0$ and $\cU^1$ have the forms given in \eqref{a11} and satisfy the profile equations \eqref{14a}, 
\eqref{14b}, and the second corrector $\cU^2_p$ remains to be chosen.

Let us denote $\cU^0_\eps(x,\theta_0):=\cU^0(x,\theta_0,\frac{x_d}{\eps})$ and similarly define $\cU^1_\eps$, 
$\cU^2_{p,\eps}$. The idea is to apply the tame estimate of Proposition \ref{i33a} to the linear singular problem 
\eqref{m22} satisfied by $W_\eps$. To apply this estimate we need the functions $\cU^0_\eps$, $\cU^1_\eps$, 
and $\cU^2_{p,\eps}$  to lie in $E^s_T$ spaces for appropriate choices of $s$. Thus, roughly speaking, these 
functions must decay in $(x',\theta_0)$ ``like functions in $H^s(x',\theta_0)$".

The pulse profiles $\sigma_{m,k}(x,\theta_m)$ defining $\cU^0(x,\theta_0,\xi_d)$ have good decay in $\theta_m$, 
so it is not hard to see  that $\cU^0_\eps$ lies in an $E^s_T$ space (Proposition \ref{h1}). Moreover, we show 
that the $\sigma_{m,k}(x,\theta_m)$ actually have moment zero, and this can be used to show that $\cU^1_\eps$, 
whose definition involves taking integrals as in \eqref{R} of an expression that is linear in $\cU^0$, also lies in an 
$E^s_T$ space (see, for example, Lemma \ref{m9} and Proposition \ref{m10a}). If the $\sigma_{m,k}(x,\theta_m)$ 
did not have moment zero, $\cU^1_\eps$ would fail to lie in any $E^s_T$ space.

However, the second corrector presents greater difficulties. Considering the $\eps^1$ profile equation in \eqref{13}, 
we see that a natural first choice for the second corrector would  be
\begin{align}\label{q3}
\cU^2 =-{\bf R}_\infty(I-{\bf E}) 
\left( \tilde L(\partial)\cU^1 +D(0)\cU^1 +({\rm d}D(0) \, \cU^0) \cU^0 \right).
\end{align}
Since the primitive $-\int_\theta^{+\infty} f(s) {\rm d}s$ of a function $f$ that decays, say, like $|s|^{-2}$ as $|s|\to+\infty$, 
itself decays to zero as $|\theta|\to+\infty$ if and only if $f$ has moment zero ($\int_\R f(s)\, {\rm d}s=0$), and 
since neither $\cU^1$ nor the  term  $( {\rm d}D(0) \, \cU^0) \cU^0$ in \eqref{q3} has moment zero\footnote{More 
precisely, we refer here to the profiles like $\tau_{m,k}(x,\theta_m)$ or products of profiles that appear in these terms.}, 
we see that this choice of $\cU^2_\eps$ generally cannot lie in any $E^s_T$ space. A natural idea then is to replace 
$\cU^0$ and $\cU^1$ by approximations that involve profiles having moment zero. More precisely, we take for 
example $\cU^0_p$ to be as in \eqref{a11}, but where the profile $\sigma_{m,k}(x,\theta_m)$ is replaced by the 
moment-zero approximation $\sigma_{m,k,p}(x,\theta_m)$ whose Fourier transform in $\theta_m$ is defined by
\begin{equation*}
\hat\sigma_{m,k,p}(x,\eta)=\chi_p(\eta)\hat\sigma_{m,k}(x,\eta).
\end{equation*}
Here $\chi_p(\eta)$ is a smooth cutoff function vanishing on a neighborhood of $0$ of size $O(|p|)$ and equal to one 
outside a slightly larger neighborhood\footnote{The cutoff renders harmless the small divisor that appears when one 
writes the Fourier transform of the $\theta_m$-primitive of $\sigma_{m,k,p}$ in terms of $\hat\sigma_{m,k,p}(x,\eta)$.}. 
In place of \eqref{q3} one could then try
\begin{align}\label{q5}
\cU^2_p =-{\bf R}_\infty \, (I-{\bf E}) \, \left( \tilde L(\partial)\cU^1_p +D(0)\cU^1_p
+({\rm d}D(0) \, \cU^0_p) \cU^0_p \right).
\end{align}
The contributions to $\cU^2_{p,\eps}$ involving $\cU^1_p$ now do lie in suitable $E^s_T$ spaces, but some of the 
contributions from the quadratic term do not lie in any $E^s_T$ space. The difficulty is caused by products of the 
form $\sigma_{m,k,p}\sigma_{m,k',p}$, and the resolution is to redefine $\cU^2_p$ as in \eqref{q5} after replacing 
these terms by $(\sigma_{m,k,p}\sigma_{m,k',p})_p$, see \eqref{k3}. We note that the contributions to $\cU^2_{p,\eps}$ 
from ``transversal interactions" like $\sigma_{m,k,p}\sigma_{m',k',p}$, where $m\neq m'$ do already lie in appropriate 
$E^s_T$ spaces, see Proposition \ref{h12}.

These moment-zero approximations introduce errors that blow up as $p\to 0$, of course, but taking $p=\eps^b$ 
for an appropriate $b>0$, one can hope to control these errors using the factor $\eps^2$ in $\eps^2\cU^2_p$. This 
strategy works and is carried out in the proof of Theorem \ref{main2} in section \ref{proofmain}. The proof relies on 
the machinery of moment-zero approximations developed in section \ref{mzero} and the quadratic interaction 
estimates of Appendix \ref{interaction}.

\begin{rem}\label{q6}
1) \textup{We learned the idea of using moment-zero approximations to construct correctors in pulse problems 
from the paper \cite{AR2}.}

2) \textup{Together with \cite{CW4}, which dealt with reflecting pulses in quasilinear uniformly stable  boundary 
problems, this paper completes the first stage of our project to rigorously justify, when that is possible, the formal 
constructions of \cite{HMR,AM,MA,MR} in boundary problems involving multiple interacting pulses. The operators 
${\bf E}$, ${\bf R}_\infty$, and the machinery of moment-zero approximations developed in these papers provide 
a set of tools for rigorously constructing leading profiles and correctors. The estimates for weakly stable singular 
systems \eqref{15p} given in \cite{CGW3}\footnote{Recall that the wavetrain estimate in \cite{CGW3} applies 
verbatim to the pulse case, as long as one uses the singular pulse calculus instead of the singular wavetrain 
calculus in the proof. Both calculi are constructed in \cite{CGW2}.} and for uniformly stable quasilinear singular 
systems in \cite{CW4} provide the basis for showing that approximate solutions are close to exact solutions for 
$\eps$ small. The approach to error analysis based on singular systems is especially well-suited to situations 
involving several pulses traveling at distinct group velocities. In these situations ``one-phase methods" that 
depend on constructing high order approximate solutions, or which use estimates of conormal or striated type, 
seem inapplicable. Finally, we stress that pulses do interact; the interactions do not produce resonances that 
affect leading order profiles as in the wavetrain case, but the interactions, whether of transversal or self-interaction 
type, do affect correctors to the leading order profiles. We have treated the analogous problems for wavetrains 
in \cite{CGW1,CGW3}.}
\end{rem}

\section{Analysis of profile equations}
\label{profanalysis}

\emph{\quad} For ease of exposition we continue as in Section \ref{derivation} to consider the $3\times 3$ strictly 
hyperbolic case. The minor changes needed to treat the general case are described in section \ref{extension}.

\subsection{The key subsystem for $\cU^0$}

Let us write $\cU^0$, the solution of the $\eps^{-1}$ equation in \eqref{13}, as
\begin{equation*}
\cU^0=\tilde\sigma_1(x,\theta_0,\xi_d) \, r_1+\tilde\sigma_2(x,\theta_0,\xi_d) \, r_2 
+\tilde\sigma_3(x,\theta_0,\xi_d) \, r_3.
\end{equation*}
As we saw in \eqref{22}, the $\tilde\sigma_i$'s have the form
\begin{equation*}
\tilde\sigma_i(x,\theta_0,\xi_d) =\sigma_i(x,\theta_0+\omega_i \, \xi_d) \quad \text{ for some } \sigma_i(x,\theta_i).
\end{equation*}
Recall that we assume that the phase $\phi_2$ is outgoing while the other two phases $\phi_1,\phi_3$ 
are incoming (the case where only one phase is incoming and all other phases are outgoing is actually simpler). 
Our first task is to show that $\sigma_2 \equiv 0$. We will need the following version of a classical lemma by 
Lax \cite{L}, see \cite[Lemma 2.11]{CGW1} for the proof.

 \begin{prop}\label{23a}
 Let $W(x,\theta_0,\xi_d)=\sum^3_{i=1}w_i(x,\theta_0,\xi_d)r_i$ be any $C^1$ function.  Then
 \begin{equation*}
 \tilde L(\partial)W=\sum^3_{i=1}(X_{\phi_i}w_i) \, r_i+ \sum^3_{i=1} \Big( \sum_{k\neq i}V^i_k w_k \Big) r_i \, ,
\end{equation*}
where $X_{\phi_i}$ is the characteristic vector field
\begin{equation*}
X_{\phi_i} :=\partial_{x_d}+\sum^{d-1}_{j=0}-\partial_{\xi_j}\omega_i(\beta)\partial_{x_j} 
=\dfrac{1}{{\bf v}_i \cdot {\bf e}_d} \, (\partial_t +{\bf v}_i \cdot \nabla_x) \, ,
\end{equation*}
and $V^i_k$, for $k\neq i$, is the tangential vector field
\begin{equation*}
V^i_k :=\sum^{d-1}_{l=0}(l_i A_l r_k) \, \partial_{x_l}.
\end{equation*}
\end{prop}


The relation \eqref{14a}(b) can be simplified by using Proposition \ref{23a}. More precisely, we recall 
the notation $F^0 :=-(\tilde L(\partial)\cU^0+D(0)\cU^0)$, see \eqref{15}. The next Proposition is an 
immediate consequence of Proposition \ref{23a}.

\begin{prop}\label{m3az}
The function $F^0$ reads $F^0 =\sum_{i=1}^3 F_i \, r_i$, with
\begin{equation*}
F_i(x,\theta_0,\xi_d)=-X_{\phi_i}\sigma_i-\sum_{k\neq i}V^i_k\sigma_k+\sum_{k=1}^3 e^i_k \, \sigma_k \, ,
\end{equation*}
for some real constants $e^i_k$ and vector fields $V^i_k$ as in Proposition \ref{23a}.
\end{prop}

In particular, the requirement ${\bf E} F^0=0$ in \eqref{14a}(b) reads
\begin{equation}
\label{equationF0}
X_{\phi}\sigma_i -e_i^i \, \sigma_i=0, \quad i=1,2,3.
\end{equation}
Since the outgoing mode $\sigma_2$ is zero in $t<0$, this implies $\sigma_2 \equiv 0$ by integrating along 
the characteristics.

Using \eqref{14} ($\eps^0$) and $\sigma_2=0$, we find the boundary condition
\begin{equation*}
B\, (\sigma_1(x',0,\theta_0) \, r_1 +\sigma_3(x',0,\theta_0) \, r_3)=0 \, ,
\end{equation*}
and we recall that the vectors $r_1,r_3$ span the stable subspace $\E^s(\beta)$, see Lemma \ref{lem1}. 
Thus, by \eqref{defe} we have
\begin{equation*}
\sigma_1(x',0,\theta_0) \, r_1 +\sigma_3(x',0,\theta_0) \, r_3 =a(x',\theta_0) \, e \, , 
\quad \text{ for some scalar function }a(x',\theta_0).
\end{equation*}
Since $e=e_1+e_3$, where $e_i\in\mathrm{span}\;\{ r_i\}$, we deduce
\begin{equation}\label{m0}
\sigma_i(x',0,\theta_0) \, r_i =a(x',\theta_0) \, e_i \, ,\quad i=1,3 \, .
\end{equation}
Using \eqref{equationF0} and \eqref{m0}, we have derived the main subsystem governing $\cU^0$:
\begin{align}\label{m1}
\begin{split}
&X_{\phi_i}\sigma_i -e_i^i \, \sigma_i=0\text{ for }i=1,3\\
&\sigma_i (x',0,\theta) \, r_i=a(x',\theta) \, e_i \, ,\\
&\sigma_i=0\text{ in } t<0 \, ,
\end{split}
\end{align}
where the scalar function $a(x',\theta)$ remains to be determined.


\paragraph{Boundary equation}

From the definition of $F^0$, the fact that $\cU^0$ is purely incoming ($\sigma_2 \equiv 0$), and Remark 
\ref{30ab}, we see that the outgoing component of the right side of \eqref{15bb}(a) is zero. Writing
\begin{equation*}
{\bf E}\cU^1=\sum^3_{i=1} \tau_i(x,\theta_0+\omega_i\xi_d) \, r_i
\end{equation*}
and using Proposition \ref{23a}, we find
\begin{equation*}
X_{\phi_2}\tau_2 -e_2^2 \, \tau_2=0, \quad \tau_2=0\text{ in }t<0 \, ,
\end{equation*}
and thus $\tau_2 \equiv 0$.

Since ${\bf E}\cU^1$ is purely incoming, it is valued in $\E^s(\beta)$. Therefore, by \eqref{defb}, the equation 
\eqref{15bb}(b) is solvable if and only if the condition
\begin{align}\label{m2}
b\cdot\left[G -({\rm d}\psi(0) \, \cU^0)\cU^0 -B \, (I-{\bf E}) \, \cU^1\right]=0 \text{ on } x_d=\xi_d=0 \, ,
\end{align}
holds. Since $b\cdot Br_i=0$, $i=1,3$, using \eqref{15} we see that
\begin{align*}
b \cdot B \, (I-{\bf E}) \cU^1 \, (x',0,\theta_0,0) &= b \cdot B \, {\bf R}_\infty F^0 \\
&= -b\cdot B \, r_2 \left( \int_{\R^+} F_2(x',0,\theta_0-\omega_2 s,s) 
\, {\rm d}s \right).
\end{align*}
Thus, the solvability condition \eqref{m2} becomes
\begin{align}\label{39}
b\cdot \left[G -({\rm d}\psi(0) \, \cU^0)\cU^0 +Br_2\left( \int_{\R^+} F_2(x',0,\theta_0-\omega_2s,s) \, {\rm d}s 
\right) \right]=0.
\end{align}

We now rewrite this equation in terms of the function $a(x',\theta_0)$ appearing in \eqref{m1}. We will use 
the following fundamental fact:

\begin{prop}[\cite{CG}]\label{m5}
There holds
\begin{equation}\label{12a}
b\cdot B\sum_{m\in \{1,3\}}R_mA_j e_m=\kappa \partial_{\eta_j}\sigma(\beta),\;j=0,\dots,d-1,\;\; 
\text{ where }\kappa\in\R\setminus 0,\;\partial_\tau \sigma(\beta)=1,
\end{equation}
the function $\sigma$ is defined in Assumption \ref{assumption3}, and $R_m:=\sum_{k\neq m} 
\frac{P_k}{\omega_m-\omega_k}$, so\footnote{Recall $P_iX=(l_i\cdot X) r_i$.}
\begin{align*}
R_1=\frac{P_2}{\omega_1-\omega_2}+\frac{P_3}{\omega_1-\omega_3},\quad 
R_3=\frac{P_1}{\omega_3-\omega_1}+\frac{P_2}{\omega_3-\omega_2}.
\end{align*}
Thus,
\begin{align}\label{12b}
b\cdot B\sum_{m\in\cI}R_m \tilde{L}(\partial)e_m=\kappa\left( 
\partial_\tau\sigma(\beta)\partial_t+\sum^{d-1}_{j=1}\partial_{\eta_j}\sigma(\beta)\partial_{x_j}\right):= X_{Lop}.
\end{align}
\end{prop}

Our main result in this paragraph reads as follows.

\begin{prop}\label{40}
There are real constants $c$ and $d$ such that taking the $\theta_0$-derivative of equation \eqref{39} is 
equivalent to the following equation for the function $a(x',\theta_0)$ in \eqref{m1}:
\begin{align}\label{41}
X_{Lop} a +c \, a +d\, \partial_{\theta_0}(a^2)=-b\cdot\partial_{\theta_0}G \, .
\end{align}
\end{prop}

\begin{proof}
\textbf{1.)} Let us first observe that $\cU^0 =a(x',\theta_0)\, e$ on $x_d=\xi_d=0$, so the bilinear term 
$({\rm d}\psi(0) \, \cU^0)\cU^0$ in \eqref{39} reads $a^2 \, ({\rm d}\psi(0) \, e)e$, which gives an $a^2$ 
contribution (up to a real multiplicative constant). This term contributes to the "Burgers like" term 
$d\, \partial_{\theta_0}(a^2)$ in \eqref{41} after taking the $\theta_0$-derivative.

\textbf{2.)} Let us recall the definition $F^0=-(\tilde L(\partial)\cU^0+D(0)\cU^0)$, and $F_2 := l_2 \cdot 
F^0$. Another easy contribution to \eqref{39} is thus given when computing 
\begin{multline*}
\int_{\R^+} l_2 \cdot D(0) (\sigma_1(x',0,\theta_0+(\omega_1-\omega_2) \, s) \, r_1 
+\sigma_3(x',0,\theta_0+(\omega_3-\omega_2) \, s) \, r_3) \, {\rm d}s \\
=\int_{\R^+} l_2 \cdot D(0) (a(x',\theta_0+(\omega_1-\omega_2) \, s) \, e_1 
+a(x',\theta_0+(\omega_3-\omega_2) \, s) \, e_3) \, {\rm d}s \, .
\end{multline*}
Taking the $\theta_0$ derivative of this integral gives a contribution of the form $c\, a(x',\theta_0)$ 
for a suitable real constant $c$.

\textbf{3.)} It remains to examine the last integral term in \eqref{39}, that is
\begin{align*}
I&:=\int_{\R^+} l_2\cdot \tilde L(\partial)\cU^0 (x',0,\theta_0-\omega_2s,s) \, {\rm d}s \\
&=\int_{\R^+} l_2\cdot \tilde L(\partial) (\sigma_1(x',0,\theta_0 +(\omega_1-\omega_2) \, s) \, r_1 
+\sigma_3(x',0,\theta_0 +(\omega_3-\omega_2) \, s) \, r_3) \, {\rm d}s \, .
\end{align*}
Since $l_2 \cdot r_i=0$, $i=1,3$, the coefficient of the $\partial_{x_d}$ term in the latter integral vanishes 
and we obtain
\begin{align}
I \, r_2=&\left[\sum^{d-1}_{j=0}(l_2A_je_1)\int_{\R^+} \partial_j a(x',\theta_0+(\omega_1-\omega_2)s) 
\, {\rm d}s\right] \, r_2 \notag\\
&+\left[\sum^{d-1}_{j=0}(l_2A_je_3) \int_{\R^+} \partial_j a(x',\theta_0+(\omega_3-\omega_2)s) 
\, {\rm d}s \right] \, r_2 \, .\label{42}
\end{align}
Let $\cA_m(x',\theta_0)$, $m=1,3$ be the unique antiderivative(s) of $a$ in $\theta_0$ such that
\begin{equation*}
\cA_m(x',(\omega_m-\omega_2)\infty)=0 \, .
\end{equation*}
The integrals in \eqref{42} can be computed and the function $I \, r_2$ can be rewritten
\begin{align}\label{42a}
I \, r_2 =\sum^{d-1}_{j=0} (l_2A_j e_1) \dfrac{\partial_j \cA_1(x',\theta_0)}{\omega_1-\omega_2} \, r_2 
+\sum^{d-1}_{j=0}(l_2A_j e_3) \dfrac{\partial_j \cA_3(x',\theta_0)}{\omega_3-\omega_2} \, r_2.
\end{align}
Using \eqref{12a}, \eqref{12b}, and the fact that $b\cdot Be_i=0$, $i=1,3$, we obtain from \eqref{42a} that 
the final contribution in taking the $\theta_0$-derivative of \eqref{39} reads
\begin{equation*}
b\cdot B\partial_{\theta_0} (I\, r_2)=X_{Lop}a.
\end{equation*}
\end{proof}

Combining \eqref{m1} with Proposition \ref{40}, we have the following key subsystem satisfied by 
$(\sigma_1,\sigma_3)$:
\begin{align}\label{m6}
\begin{split}
&(a)\;X_{\phi_1}\sigma_1 -e_1^1 \, \sigma_1=0\\
&(b)\;X_{\phi_3}\sigma_3 -e_3^3 \, \sigma_3=0\\
&(c)\;X_{Lop}a +c \, a +d\, \partial_{\theta_0}(a^2)  =-b\cdot \partial_{\theta_0}G \text{ on }\{x_d=0,\xi_d=0\}.
\end{split}
\end{align}
where
\begin{align}\label{m7}
\sigma_i(x',0,\theta_0) r_i=a(x',\theta_0)e_i
\end{align}
and all unknowns are zero in $t<0$. Recall that the outgoing component $\sigma_2$ of $\cU^0$ vanishes. 
In Proposition \ref{m8} we solve for $a$ by Picard iteration, and then obtain $\sigma_1$ and $\sigma_3$ 
in the interior by simply integrating along the appropriate characteristics.

\begin{rem}
\textup{In the case of wavetrains the interior equation for $\tau_2$ is coupled to $\sigma_1$ and $\sigma_3$ 
by a nonlinear interaction integral whenever there is a resonance of the form
\begin{align*}
n\phi_1+n_3\phi_3=n_2\phi_2, \text{ with }n_i\in\Z.
\end{align*}
Thus, $\tau_2$, the outgoing component of ${\bf E}\cU^1$ is not necessarily zero, and this leads to the presence 
of an extra term of the form $\partial_{\theta_0}\tau_2 $ in the last equation of \eqref{m6}. That extra term 
leads to a loss of derivatives in the estimate for the linearized system, and forced us to use Nash-Moser 
iteration to solve for $(\sigma_1,\sigma_3,\tau_2)$ in \cite{CGW3}.}
\end{rem}

We shall work with profiles lying in a class of Sobolev spaces weighted in $\theta$. Let
\begin{align*}
\Omega_T=\{(x,\theta)\in \overline{\mathbb{R}}^{d+1}_+ \times\R: t<T\},
\end{align*}
and introduce the following spaces.

\begin{defn}\label{c1}
For $s\in\N$, define the spaces
\begin{equation}
\Gamma^s_T:=\{a(x,\theta)\in L^2(\Omega_T): (\theta,\partial_{x},\partial_{\theta})^\beta a\in L^2(\Omega_T) 
\text{ for }|\beta|\leq s\text{ and }a=0\text { in }t<0\} \, ,
\end{equation}
with norms
\begin{align}\label{c1a}
|a|_{s,T}=\sum_{|\beta|=|\beta_1,\beta_2,\beta_3|\leq s} 
|\theta^{\beta_1}\partial_x^{\beta_2}\partial_\theta^{\beta_3}a|_{L^2(\Omega_T)}.
\end{align}
The analogous norms of functions of $b(x',\theta)$ defined  on $b\Omega_T:=\{(x',\theta)\in \mathbb{R}^d 
\times \R:t<T\}$ are denoted with brackets: $\langle b\rangle_{s,T}$.  We denote the corresponding spaces 
by $b\Gamma^s_T$.

We will let $H^s_T$ denote the usual Sobolev space on $\Omega_T$ with norm defined just as in \eqref{c1a} 
but without the $\theta$ weights.  These spaces and those below have the obvious meanings when a function 
is vector-valued.
\end{defn}

\begin{rem}\label{f3}
\textup{By Sobolev embedding if the functions $f^i_k$, $g^i_{l,m}$, $h^i_{l,m}$ appearing in \eqref{30a} lie in 
$\Gamma^s_T$ for $s>\frac{d+2}{2}+3$, then $F$ as in \eqref{30aa} is of type $\cF$.}
\end{rem}

\begin{prop}\label{m8}
Let $T>0$, $m>\frac{d+1}{2}+1$ and suppose $G\in b\Gamma^{m+1}_T$ and vanishes in $t<0$. Then provided 
$\langle G\rangle_{b\Gamma^{m+1}_T}$ is small enough, the system \eqref{m6}, \eqref{m7} has a unique solution 
satisfying $a \in b\Gamma^m_{T}$ and  $\sigma_i\in\Gamma^m_{T}$, $i=1,3$.
\end{prop}

\begin{proof}
First one solves \eqref{m6}(c) using the iteration scheme
\begin{equation*}
X_{Lop}a^{n+1} +2\, d \, a^n\partial_{\theta_0}a^{n+1}  =-b\cdot \partial_{\theta_0}G -c \, a^n.
\end{equation*}
The standard proof showing convergence towards a solution $a\in H^m_T(x',\theta_0)$ is easily adapted to the 
weighted spaces $b\Gamma^m_T$. Details of a very similar argument are given in the proof of \cite[Proposition 3.6]{CW4}. 
The conclusion for $\sigma_1$ and $\sigma_3$ follows easily.
\end{proof}

\subsection{The first corrector $\cU^1$}

\emph{\quad} To understand the properties of $\cU^1$ we need:

\begin{lem}\label{m9}
The functions $\sigma_1$, $\sigma_3$, and $a$ all have moment zero.
\end{lem}

\begin{proof}
Let $\ua(x)=\int_\R a(x',\theta_0) \, {\rm d}\theta_0$. Taking the moment of equation \eqref{m6}(c), we obtain 
the problem for $\ua$:
\begin{equation*}
X_{Lop}\ua+c_4\ua=0,\;\;\ua=0\text{ in }t<0.
\end{equation*}
Thus $\ua=0$ and so the equations for $\sigma_i$, $i=1,3$ imply their moments are zero as well.
\end{proof}

\begin{prop}\label{m10a}
Let $\sigma(x,\theta)\in\Gamma^s_T$, $s>\frac{d}{2}+3$, have moment zero and let $\sigma^*(x,\theta)$ 
denote the unique $\theta$-primitive of $\sigma$ that decays to zero as $|\theta|\to\infty$. Then $\sigma^* 
\in\Gamma^{s-1}_T$.
\end{prop}

\begin{proof}
We have $\sigma\in\Gamma^s_T$ if and only if $\partial_m^k \partial_x^\alpha (m^p \, \hat\sigma(x,m)) \in 
L^2(x,m)$ for all $k,\alpha,p$ such that $k+|\alpha|+p\leq s$. Since $\hat \sigma(x,0)=0$ we have
\begin{equation*}
\hat\sigma(x,m)=m\int^1_0(\partial_m\hat\sigma) (x,sm) \, {\rm d}s
\end{equation*}
so
\begin{equation*}
\widehat{\sigma^*}(x,m)=\dfrac{1}{im} \, \hat\sigma(x,m) =-i \, \int^1_0(\partial_m\hat\sigma) (x,sm) \, 
{\rm d}s,
\end{equation*}
and the result follows.
\end{proof}

\begin{prop}\label{m11a}
Suppose $G\in b\Gamma^{s+1}_T$, $s>\frac{d}{2}+4$, and vanishes in $t<0$. Then the first corrector 
$\cU^1$ in \eqref{14b} satisfies
\begin{multline}\label{m11aa}
(I-{\bf E})\cU^1 (x,\theta_0,\xi_d)=\sum^3_{i=1}W_i(x,\theta_0,\xi_d) \, r_i \, ,\\
\text{\rm where } W_i(x,\theta_0,\xi_d)=\sum_{k\neq i}b_k(x,\theta_0+\omega_k\xi_d) \, ,\quad 
\text{\rm with  } b_k\in\Gamma^{s-2}_T.
\end{multline}
\end{prop}

\begin{proof}
We have $(I-{\bf E})\cU^1={\bf R}_\infty F^0={\bf R}_\infty (I-{\bf E})F^0$, so using Proposition \ref{15} we obtain
\begin{align*}
((I-{\bf E})\cU^1)_i&= -\sum_{k\neq i} e^i_k \, \int_{\xi_d}^{+\infty} 
\sigma_k(x,\theta_0+\omega_i\xi_d+s(\omega_k-\omega_i)) \, {\rm d}s\\
&+\sum_{k\neq i} \int_{\xi_d}^{+\infty} V^i_k \sigma_k(x,\theta_0+\omega_i\xi_d+s(\omega_k-\omega_i)) 
\, {\rm d}s \, .
\end{align*}
Now $\sigma_k\in\Gamma^s_T$ has moment zero, so
\begin{align*}
\int_{\xi_d}^{+\infty}Ê\sigma_{k}(x,\theta_0+\omega_i\xi_d+s(\omega_k-\omega_i)) \, {\rm d}s 
=\sigma_k^*(x,\theta_0+\omega_k\xi_d)/(\omega_k-\omega_i),
\end{align*}
where $\sigma_k^*(x,\theta_k)\in \Gamma^{s-1}_T$ by Proposition \ref{m10a}. The other terms are treated 
similarly using $V^i_k\sigma_k\in\Gamma^{s-1}_T$.
\end{proof}

\begin{prop}\label{m12a}
Suppose $G\in b\Gamma^{s+1}_T$, $s>\frac{d}{2}+4$, and vanishes in $t<0$. Then there exists a solution 
$\cU^1$ to \eqref{14b} that satisfies
\begin{equation*}
{\bf E}\cU^1 (x,\theta_0,\xi_d) =\tau_1(x,\theta_0+\omega_1\xi_d)r_1 +\tau_3(x,\theta_0+\omega_3\xi_d)r_3 
\quad \text{\rm with  } \tau_k\in\Gamma^{s-3}_T.
\end{equation*}
\end{prop}

\begin{proof}
We have already shown $\tau_2=0$.  Since the $\sigma_i\in\Gamma^s_T$, $i=1,3$, have moment zero, and 
the functions $b_k$ in \eqref{m11aa} lie in $\Gamma^{s-2}_T$, it follows that the right side of \eqref{15bb}(a) 
has components given by functions in $\Gamma^{s-3}_T$. Similarly, the right side of the boundary equation 
\eqref{15bb}(b) has components given by functions in $b\Gamma^{s-3}_T$. The solvability condition \eqref{m2} 
holds, so we may make a choice of boundary data ${\bf E}\cU^1(x',0,\theta_0,0)$ satisfying \eqref{15bb}(b) 
whose components are elements of $b\Gamma^{s-3}_T$. Thus, solving the system \eqref{15bb}, we obtain 
a solution ${\bf E}\cU^1$ with components given by $\tau_k\in\Gamma ^{s-3}_T$, $k=1,3$.
\end{proof}

\begin{defn}\label{m13ab}
Suppose $W(x,\theta_0,\xi_d)={\bf E} W +(I-{\bf E})W$ and
\begin{align*}
{\bf E}W(x,\theta_0,\xi_d) &= 
\sum^3_{i=1}a_i(x,\theta_0+\omega_i\xi_d)r_i,  \;\text{\rm with }a_i(x,\theta_i)\in\Gamma^s_T;\\
(I-{\bf E})W(x,\theta_0,\xi_d) &= \sum^3_{i=1}w_ir_i, \, \, \text{\rm where }w_i(x,\theta_0,\xi_d) 
= \sum_{k\neq i}b_k(x,\theta_0+\omega_k\xi_d) \, \, \text{\rm with  }b_k(x,\theta_k)\in\Gamma^{s}_T.
\end{align*}
Then we write $W\in\tilde\Gamma^s_T$.
\end{defn}

We may summarize our construction of $\cU^0$ and $\cU^1$ as follows.

\begin{theo}
\label{m14ab}
Suppose $G\in b\Gamma^{s+1}_T$, $s>\frac{d}{2}+4$, and vanishes in $t<0$. Then provided 
$\langle G\rangle_{b\Gamma^{s+1}_T}$ is small enough, there exist solutions $\cU^0\in\tilde\Gamma^s_T$, 
$\cU^1\in\tilde\Gamma^{s-3}_{T}$, satisfying the profile equations \eqref{14a}, \eqref{14b}. Moreover, 
$\cU^0$ and ${\bf E}\cU^1$ are purely incoming.
\end{theo}

\section{The exact solution of the singular system}
\label{exact}

\emph{\quad} In this section we state the main estimates for the linearized singular system and also the existence 
theorem for the nonlinear singular system \eqref{15p}.  We begin by gathering the notation for spaces and norms 
that is needed below.

\begin{nota}
\label{spaces}
Here  we take $s\in \N=\{0,1,2,\dots\}$.

(a)\; Let $\Omega:=\overline{\mathbb{R}}^{d+1}_+\times\mathbb{R}$, $\Omega_T:=\Omega\cap\{-\infty<t<T\}$, 
$b\Omega:=\mathbb{R}^d\times\mathbb{R}$, $b\Omega_T:=b\Omega\cap \{-\infty<t<T\}$,  and set $\omega_T 
:=\overline{\mathbb{R}}^{d+1}_+\cap\{-\infty<t<T\}$.

(b)\; Let $H^s\equiv H^s(b\Omega)$, the standard Sobolev space with norm $\langle V(x',\theta_0)\rangle_s$. 
For $\gamma\geq 1$ we set $H^s_\gamma:=e^{\gamma t} \, H^s$ and $\langle V\rangle_{s,\gamma} := \langle 
e^{-\gamma t} \, V \rangle_s$.

(c)\; $L^2H^s\equiv L^2(\overline{\mathbb{R}}_+,H^s(b\Omega))$ with norm $|U(x,\theta_0)|_{L^2H^s} \equiv 
|U|_{0,s}$ given by
\begin{equation*}
|U|_{0,s}^2=\int^\infty_0|U(x',x_d,\theta_0)|_{H^s(b\Omega)}^2dx_d.
\end{equation*}
The corresponding norm on $L^2H^s_\gamma$ is denoted $|V|_{0,s,\gamma}$.

(d)\; $CH^s\equiv C(\overline{\mathbb{R}}_+,H^s(b\Omega))$ denotes the space of continuous bounded 
functions of $x_d$ with values in $H^s(b\Omega)$, with norm $|U(x,\theta_0)|_{CH^s} =|U|_{\infty,s} := 
\sup_{x_d\geq 0} |U(.,x_d,.)|_{H^s(b\Omega_T)}$.
The corresponding norm on $CH^s_\gamma$ is denoted $|V|_{\infty,s,\gamma}$.

(e)\; Let $M_0:=3d+5$ and define $C^{0,M_0} :=C(\overline{\mathbb{R}}_+,C^{M_0}(b\Omega))$ as 
the space of continuous bounded functions of $x_d$ with values in $C^{M_0}(b\Omega)$, with norm 
$|U(x,\theta_0)|_{C^{0,M_0}}:= |U|_{L^\infty W^{M_0,\infty}}$. Here $L^\infty W^{M_0,\infty}$ denotes 
the space $L^\infty(\overline{\R}_+;W^{M_0,\infty}(b\Omega))$\footnote{The size of $M_0$ is determined 
by the requirements of the singular calculus described in Appendix \ref{calculus}.}.

(f)\; The corresponding spaces on $\Omega_T$ are denoted $L^2H^s_T$, $L^2H^s_{\gamma,T}$, $CH^s_T$, 
$CH^s_{\gamma,T}$ and $C^{0,M_0}_T$ with norms $|U|_{0,s,T}$, $|U|_{0,s,\gamma,T}$, $|U|_{\infty,s,T}$, 
$|U|_{\infty,s,\gamma,T}$, and $|U|_{C^{0,M_0}_T}$ respectively. On $b\Omega_T$ we use the spaces 
$H^s_T$ and $H^s_{\gamma,T}$ with norms $\langle U\rangle_{s,T}$ and $\langle U\rangle_{s,\gamma,T}$.

(g)\; All constants appearing in the estimates below are independent of $\eps$, $\gamma$, and $T$ unless 
such dependence is explicitly noted.
\end{nota}

The following spaces appear in the statement of the main existence theorem and are used in its proof.

\begin{defn}
\label{Espaces}
For $s\in \{ 0,1,2,\dots \}$, let
\begin{align*}
E^s_T &:=CH^s_T\cap L^2H^{s+1}_T \, ,\quad \text{\rm with the norm } 
|U(x,\theta_0)|_{E^s_T}:=|U|_{\infty,s,T}+|U|_{0,s+1,T} \, ,\\
E^s_{\gamma,T} &:=CH^s_{\gamma,T}\cap L^2H^{s+1}_{\gamma,T}\, ,\quad \text{\rm with the norm }
|U(x,\theta_0)|_{E^s_{\gamma,T}} := |U|_{\infty,s,\gamma,T} +|U|_{0,s+1,\gamma,T} \, .
\end{align*}
 \end{defn}

The linearization of the singular problem \eqref{15p} at $U(x,\theta_0)$ has the form
\begin{align}\label{i3}
\begin{split}
&(a)\, \partial_d \dot U_\eps +\mathbb{A} \left( \partial_{x'}+\dfrac{\beta \partial_{\theta_0}}{\eps} \right)
\dot U_\eps +\cD (\eps U) \, \dot U_\eps=f(x,\theta_0) \quad \text{ on }\Omega \, ,\\
&(b)\, \cB (\eps U) \, \dot U_\eps|_{x_d=0} =g(x',\theta_0) \, ,\\
&(c)\, \dot U_\eps=0 \text{ in } t<0,
\end{split}
\end{align}
where the matrices $\cB(\eps U)$, $\cD(\eps U)$ are defined in \eqref{caaa}\footnote{Here and below we often 
suppress the subscript $\eps$ on $\dot U$.}. Instead of \eqref{i3}, consider the equivalent problem satisfied by 
$\dot U^\gamma :=e^{-\gamma t}\dot U$:
\begin{align}\label{i5}
\begin{split}
&\partial_d \dot U^\gamma +\mathbb{A} \left( (\partial_{t}+\gamma,\partial_{x''})
+\dfrac{\beta \, \partial_{\theta_0}}{\eps} \right) \dot U^\gamma +\cD (\eps U) \, \dot U^\gamma
=f^\gamma(x,\theta_0) \, ,\\
&\cB (\eps U) \, \dot U^\gamma|_{x_d=0} =g^\gamma(x',\theta_0) \, ,\\
&\dot U^\gamma=0 \text{ in } t<0 \, .
\end{split}
\end{align}
Below we let $\Lambda_D$ denote the singular Fourier multiplier (see \eqref{singularpseudop}) associated 
with the symbol
\begin{equation*}
\Lambda(X,\gamma) :=\left( \gamma^2+\left|\xi'+\frac{k\, \beta}{\eps}\right|^2 \right)^{1/2},\;
X := \xi'+\dfrac{k\, \beta}{\eps}.
\end{equation*}
The basic  estimate for the linearized singular problem \eqref{i5} is given in the next Proposition. Observe that the 
estimate \eqref{aprioriL2} exhibits a loss of one ``singular derivative" $\Lambda_D$. This is quite a high price to pay, 
which counts as a factor $1/\eps$. In view of \cite[Theorem 4.1]{CG}, there is strong evidence that the loss below is 
optimal.

\begin{prop}[Main $L^2$ linear estimate]\label{i5z}
We make the structural assumptions of Theorem \ref{main}, let $s_0:=\left[\frac{d+1}{2}\right]+1$,  and recall 
$M_0 =3d+5$. Fix $K>0$ and suppose $|\eps \, \partial_d U|_{C^{0,M_0-1}} +|U|_{C^{0,M_0}} +|U|_{CH^{s_0}} 
\leq K$ for $\eps \in (0,1]$. There exist positive constants $\eps_0(K)>0$, $C(K)>0$ and $\gamma_0(K) \ge 1$ 
such that sufficiently smooth solutions $\dot U$ of the linearized singular problem \eqref{i3} satisfy\footnote{Note 
that the norms $|u|_{0,1}$ and $|\Lambda_D u|_{0,0}$ are not equivalent.}:
\begin{align}\label{aprioriL2}
|\dot U^\gamma|_{0,0} +\dfrac{\langle \dot U^\gamma\rangle_0}{\sqrt{\gamma}} 
\leq C(K) \left( \dfrac{|\Lambda_D f^\gamma|_{0,0}+|f^\gamma/\eps|_{0,0}}{\gamma^2} 
+\dfrac{\langle\Lambda_D g^\gamma \rangle_0 +\langle g^\gamma/\eps \rangle_0}{\gamma^{3/2}} \right) 
\end{align}
for $\gamma \geq \gamma_0(K)$, \;$0<\eps\leq \eps_0(K)$.

The same estimate holds if $\cB(\eps U)$ in \eqref{i3} is replaced by $\cB(\eps U,\eps\cU)$ and $\cD(\eps U)$ 
is replaced by $\cD(\eps U,\eps\cU)$ as long as $|\eps \partial_d(U,\cU)|_{C^{0,M_0-1}} +|U,\cU|_{C^{0,M_0}} 
+|U,\cU|_{CH^{s_0}} \leq K$ for $\eps\in (0,1]$.
\end{prop}

The proof is identical to the proof of Proposition 2.2 in \cite{CGW3}. Here our hypothesis on $|U|_{CH^{s_0}}$ is 
needed to allow the pulse calculus summarized in Appendix \ref{calculus} to be used in exactly the same way that 
the wavetrain calculus was used in \cite{CGW3}. The corollary below follows from Proposition \ref{i5z} by the same 
argument used to derive Corollary 2.3 from Proposition 2.2 in \cite{CGW3}.

\begin{cor}[Main $H^1_{tan}$ linear estimate]\label{estimH1}
Under the same assumptions as in Proposition \ref{i5z}, smooth enough solutions $\dot U$ of the linearized singular 
problem \eqref{i3} satisfy:
\begin{align}\label{i6}
|\dot U^\gamma|_{\infty,0}+|\dot U^\gamma|_{0,1}+\frac{\langle \dot U^\gamma\rangle_1}{\sqrt{\gamma}} 
\leq C(K) \left( \frac{|\Lambda_D f^\gamma|_{0,1}+|f^\gamma/\eps|_{0,1}}{\gamma^2} 
+\frac{\langle\Lambda_D g^\gamma \rangle_1 +\langle g^\gamma/\eps \rangle_1}{\gamma^{3/2}} \right) 
\end{align}
for $\gamma \geq \gamma_0(K)$, \;$0<\eps\leq \eps_0(K)$.
\end{cor}

The next proposition localizes the estimate \eqref{i6} to $\Omega_T$. By considering data $\eps f$, $\eps g$ 
instead of $f,g$ we can recast \eqref{i6} in a form where the loss of a singular derivative is replaced by loss of 
an ordinary derivative. The proof is identical to that of \cite[Proposition 2.8]{CGW3}. Let us write the linearized 
operators on the left sides of \eqref{i3}(a) and (b) as $\bL'(\eps U)\dot U$ and $\bB'(\eps U) \dot U$ respectively.

\begin{prop}\label{i14}
Let  $s_0=\left[\frac{d+1}{2}\right]+1$, fix $K>0$, and suppose $|\eps \partial_d U|_{C^{0,M_0-1}_T} 
+|U|_{C^{0,M_0}_T}+|U|_{CH^{s_0}_T}\leq K$ for $\eps\in (0,1]$. There exist positive constants $\eps_0(K)$, 
$\gamma_0(K)$ such that solutions of the singular problem
\begin{align*}
\begin{split}
&\bL'_\eps(U)\dot U=\eps f \quad \text{ in } \Omega_T \, ,\\
&\bB'_\eps(U)\dot U=\eps g \quad \text{ on } b\Omega_T \, ,\\
&\dot U=0 \quad \text{ in }t<0,
\end{split}
\end{align*}
satisfy
\begin{equation*}
|\dot U|_{\infty,0,\gamma,T} +|\dot U|_{0,1,\gamma,T} 
+\dfrac{\langle \dot U|_{x_d=0} \rangle_{1,\gamma,T}}{\sqrt{\gamma}} \le 
C(K) \, \left( \dfrac{|f|_{0,2,\gamma,T}}{\gamma^2} +\dfrac{\langle g \rangle_{2,\gamma,T}}{\gamma^{3/2}} \right) 
\end{equation*}
for $0<\eps \leq \eps_0(K)$, $\gamma \geq \gamma_0(K)$, and the constant $C(K)$ only depends on $K$.

The same estimate holds if $\cB(\eps U)$ in \eqref{i3} is replaced by $\cB(\eps U,\eps\cU)$ given in \eqref{ca2}, 
and $\cD(\eps U)$ is replaced by $\cD(\eps U,\eps\cU)$ given in \eqref{caa}, as long as there holds $|\eps 
\partial_d(U,\cU)|_{C^{0,M_0-1}_T} +|U,\cU|_{C^{0,M_0}_T}+|U,\cU|_{CH^{s_0}_T}\leq K$ for $\eps\in (0,1]$.
\end{prop}

We need the following higher derivative estimate in the error analysis. This estimate is also needed in the Nash-Moser 
iteration used to prove the nonlinear existence Theorem \eqref{k8b}. We recall the definitions \eqref{a11z} for the 
indices $a_0,a_1,a,\tilde{a}$.

\begin{prop}[Tame estimate for the linearized system]
\label{i33a}
Fix $K>0$  and suppose
\begin{align}\label{i33az}
|\eps \, \partial_dU|_{C^{0,M_0-1}_T} +|U|_{C^{0,M_0}_T}+|U|_{CH^{s_0}_T}\leq K \text{ for }\eps\in (0,1].
\end{align}
Let  $\mu_0 := a_0+2$ and $s\in[0,\tilde a]$, where $\tilde{a}$ is defined in \eqref{a11z}. There exist positive 
constants $\gamma=\gamma(K)$, $\kappa_0(\gamma,T)$, $\eps_0$, and $C$ such that if
\begin{equation*}
|U|_{0,\mu_0,\gamma,T} +\langle U|_{x_d=0} \rangle_{\mu_0,\gamma,T} \le \kappa_0,
\end{equation*}
then solutions $\dot U$ of the linearized system \eqref{i3} satisfy for $0<\eps \leq \eps_0$:
\begin{multline*}
|\dot U|_{E^s_{\gamma,T}} +\langle \dot U|_{x_d=0} \rangle_{s+1,\gamma,T} \\
\le C\, \Big[ |f|_{0,s+2,\gamma,T} +\langle g\rangle_{s+2,\gamma,T} +\left( |f|_{0,\mu_0\gamma,T} 
+\langle g\rangle_{\mu_0,\gamma,T} \right) \, \left(|U|_{0,s+2,\gamma,T} 
+\langle U|_{x_d=0} \rangle_{s+2,\gamma,T} \right) \Big].
\end{multline*}
The same estimate holds for $0<\eps\leq \eps_0$ if $\cB(\eps U)$ in \eqref{i3} is replaced by $\cB(\eps U,\eps\cU)$ 
and $\cD(\eps U)$ is replaced by $\cD(\eps U,\eps\cU)$ as long as $|\eps \partial_d(U,\cU)|_{C^{0,M_0-1}_T} 
+|U,\cU|_{C^{0,M_0}_T}+|U,\cU|_{CH^{s_0}_T} \leq K$ for $\eps\in (0,1]$ and
\begin{equation*}
|U,\cU|_{0,\mu_0,\gamma,T} +\langle U|_{x_d=0},\cU|_{x_d=0} \rangle_{\mu_0,\gamma,T} \le \kappa_0.
\end{equation*}
\end{prop}

The proof is exactly the same as the proof of Proposition 2.15  in \cite{CGW3}. Finally, we state the existence 
theorem for solutions of the singular system \eqref{15p}.

\begin{theo}\label{k8b}
Fix $T>0$, define $a$, $a_0$, and $\tilde a$ as in \eqref{a11z}, and suppose $G\in H^{\tilde a} (b\Omega_T)$. 
There exists $\eps_0>0$ such that if $\langle G\rangle_{a+2}$ is small enough, there exists a solution $U_\eps$ 
of the system  \eqref{15p} on $\Omega_T$ for $0<\eps\leq \eps_0$ with $U_\eps\in E^{a-1}$, $U_\eps|_{x_d=0} 
\in H^a$. This statement remains true if $a$ is increased and if $\tilde a\geq 2a-a_0$.
\end{theo}

\begin{proof}
The proof is an exact repetition of the Nash-Moser argument used to prove Theorem 5.13 in \cite{CGW3}. 
One fixes $T>0$, $K>0$ and $\gamma=\gamma(K)$ as in Proposition \ref{i33a}, and uses the scale of spaces 
$E^s_{\gamma,T}$ (Definition \ref{Espaces}) with $\gamma \, T=1$. The smoothing operators $S_\theta$ are 
defined just as in Lemma 5.12 of \cite{CGW3}, taking account in the obvious way for the fact that now functions 
$u(x,\theta_0) \in E^s_{\gamma,T}$ are defined for $\theta_0 \in \R$ instead of $\theta_0 \in {\mathbb T}$. The 
verification of the condition \eqref{i33az} needed to apply the tame estimate of Proposition \ref{i33a} in the 
induction argument is done just as in Lemma 5.20 of \cite{CGW3}.
\end{proof}

\section{Error analysis}
\label{error}

\subsection{Moment-zero approximations}\label{mzero}

\quad When constructing a corrector to the leading part of the approximate solution we must take primitives in 
$\theta$ of functions $f(x,\theta)$ that decay to zero as $|\theta|\to\infty$. A difficulty is that such primitives do 
not necessarily decay to zero as $|\theta|\to \infty$, and this prevents us from using those primitives directly in 
the error analysis. The failure of the primitive to decay manifests itself on the Fourier transform side as a small 
divisor problem. To get around this difficulty we work with the primitive of a \emph{moment-zero approximation}, 
because such a primitive does have the desired decay.

We will use  the following spaces:

\begin{defn}\label{d22}
1.)  For $s\geq 0$, let $E^s_T$ be given in Definition \ref{Espaces}.


2.)  Let $\cE^s_T:=\{\cU(x,\theta_0,\xi_d):|\cU|_{\cE^s_T}:=sup_{\xi_d\geq 0}|\cU(\cdot,\cdot,\xi_d)|_{E^s_T} < \infty\}$.
\end{defn}


\begin{prop}\label{d23a}
For $s>(d+1)/2$, the spaces $E^s_T$ and $\cE^s_T$ are Banach algebras.
\end{prop}


\begin{proof}
This is a consequence of Sobolev embedding and the fact that $L^\infty(b\Omega_T)\cap H^s(b\Omega_T)$ 
is a Banach algebra for $s\geq 0$.
\end{proof}

The proofs of the following two propositions follow directly from the definitions.

\begin{prop}\label{h1}
(a)\;For $s\geq 0$, let $\sigma(x,\theta)\in H^{s+1}_T$ and set $\tilde \sigma(x,\theta_0,\xi_d):=\sigma 
(x,\theta_0+\omega\xi_d)$, $\omega\in\R$. Then there holds
\begin{align*}
|\tilde\sigma|_{\cE^s_T}\leq C|\sigma|_{H^{s+1}_T}.
\end{align*}

(b) Set $\tilde \sigma_\eps(x,\theta_0):=\tilde\sigma(x,\theta_0,\frac{x_d}{\eps})$.  Then
\begin{align*}
|\tilde\sigma_\eps|_{E^s_T} \leq |\tilde\sigma|_{\cE^s_T}.
\end{align*}
\end{prop}

\begin{defn}[Moment-zero approximations]\label{h5}
Let $0<p \le 1$, and let $\phi\in C^\infty(\R)$ have $\mathrm{supp}\;\phi\subset \{m:|m|\leq 2\}$, with $\phi=1$ 
on $|m|\leq 1$. Set $\phi_p(m)=\phi(\frac{m}{p})$ and $\chi_p=1-\phi_p$. For $\sigma(x,\theta)\in L^2$, define 
the \emph{moment zero approximation} to $\sigma$, $\sigma_p(x,\theta)$ by
\begin{align}\label{h5a}
\hat\sigma_p(x,m):=\chi_p(m) \, \hat\sigma(x,m),
\end{align}
where the hat denotes the Fourier transform in $\theta$.
\end{defn}

\begin{prop}\label{h6}
For $s\geq 1$ suppose $\sigma(x,\theta)\in \Gamma^{s+2}$ and $\tilde \sigma(x,\theta_0,\xi_d):=\sigma 
(x,\theta_0+\omega\xi_d)$. Then
\begin{align*}
\begin{split}
&(a) \, \, |\tilde\sigma-\tilde \sigma_p|_{\cE^s_T}\leq C \, |\sigma|_{\Gamma^{s+2}_T} \, \sqrt{p},\\
&(b) \, \, |\partial_{x_d}\tilde\sigma-\partial_{x_d}\tilde \sigma_p|_{\cE^{s-1}_T}\leq C \, |\sigma|_{\Gamma^{s+2}_T} 
\, \sqrt{p}.
\end{split}
\end{align*}
\end{prop}

\begin{proof}
\textbf{1.) } Recall that $\sigma \in \Gamma^s_T \Leftrightarrow \theta^{\beta_1} \partial_x^{\beta_2} 
\partial_\theta^{\beta_3} \sigma(x,\theta) \in L^2(x,\theta)$ for $|\beta|\leq s\Leftrightarrow \partial_m^{\beta_1} 
\partial_x^{\beta_2} m^{\beta_3} \hat\sigma(x,m) \in L^2(x,m)$ for $|\beta|\leq s$. It follows that
\begin{align}\label{h7}
\sigma\in\Gamma^{s+2}_T\Rightarrow \hat \sigma(x,m)\in H^{s+2}_T(x,m)\subset H^1(m,H^{s+1}(x)) 
\subset L^\infty(m,H^{s+1}(x)).
\end{align}


\textbf{2.) }We have
\begin{align*}
\begin{split}
&|\sigma-\sigma_p|^2_{H^{s+1}_T} \sim \sum_{|\alpha|+k\leq s+1} \Big| \partial_x^\alpha m^k \hat\sigma(x,m) 
(1-\chi_p(m)) \Big|^2_{L^2(x,m)} \\
&=\sum_{|\alpha|+k\leq s+1} \int_{|m|\leq 2p} \int |\partial_x^\alpha m^k\hat\sigma(x,m)(1-\chi_p(m))|^2 \, {\rm d}x 
\, {\rm d}m \\
&\leq C \, \int_{|m|\leq 2p} |\hat\sigma(x,m)|^2_{H^{s+1}(x)} \, {\rm d}m \leq C \, |\sigma|^2_{\Gamma^{s+2}_T} (2p),
\end{split}
\end{align*}
where the last inequality uses \eqref{h7}. The conclusion now follows from Propostion \ref{h1}.

\textbf{3.) }The proof of part $(b)$ is essentially the same.
\end{proof}

\begin{prop}\label{h8}
Let $\sigma(x,\theta)\in H^s_T$, $s\geq 0$, and let $\sigma_p$ be a moment-zero approximation to $\sigma$. We have
\begin{align*}
\begin{split}
&(a) \, \, |\sigma_p|_{H^s_T} \leq C \, |\sigma|_{H^s_T} \, ,\\
&(b) \, \, \text{If }\sigma\in\Gamma^s_T, \text{ then }|\sigma_p|_{\Gamma^s_T} \leq \dfrac{C}{p^s} \, 
|\sigma|_{\Gamma^s_T}.
\end{split}
\end{align*}
\end{prop}

\begin{proof}
Part $(b)$ follows from \eqref{h5a}. Indeed, for $|\beta|\leq s$, there holds
\begin{align*}
|\partial_m^{\beta_1}\partial_x^{\beta_2}m^{\beta_3}\hat\sigma_p(x,m)|_{L^2_T} \leq \frac{C}{p^{\beta_1}} \, 
|\sigma|_{\Gamma^s_T} \, ,
\end{align*}
since $|\partial_m^{\beta_1}\chi_p|\leq C/p^{\beta_1}$. Taking $\beta_1=0$ we similarly obtain part $(a)$.
\end{proof}

Next we consider primitives of moment-zero approximations.

\begin{prop}\label{h9}
Let $\sigma(x,\theta)\in\Gamma^s_T$, $s>\frac{d}{2}+3$. Let $\sigma_p^*(x,\theta)$ be the unique primitive of 
$\sigma_p$ in $\theta$ that decays to zero as $|\theta|\to \infty$. Then $\sigma^*_p\in\Gamma^s_T$ with moment 
zero, and
\begin{align}\label{h10}
\begin{split}
&(a) \, \, |\sigma^*_p|_{H^s_T}\leq C \, \frac{|\sigma_p|_{H^s_T}}{p}\\
&(b) \, \, |\sigma^*_p|_{\Gamma^s_T}\leq C \, \frac{|\sigma_p|_{\Gamma^s_T}}{p^{s+1}}.
\end{split}
\end{align}
\end{prop}

\begin{proof}
\textbf{1.) }Since $\sigma_p(x,\theta)\in\Gamma^s_T$, $s>\frac{d}{2}+3$, we have $|\sigma_p(x,\theta)|\leq C \, 
\langle\theta\rangle^{-2}$ for all $(x,\theta)$. The unique $\theta$-primitive of $\sigma_p$ decaying to zero as 
$|\theta|\to \infty$ is thus
\begin{align*}
\sigma^*_p(x,\theta)=-\int^{+\infty}_\theta \sigma_p(x,s) \, {\rm d}s =\int^\theta_{-\infty}\sigma_p(x,s) \, {\rm d}s.
\end{align*}
Moreover, we have
\begin{align}\label{h11}
\partial_\theta\sigma^*_p=\sigma_p \Leftrightarrow im \, \widehat{\sigma^*_p}=\widehat{\sigma_p}=\chi_p \, 
\hat\sigma, \text{ so }\widehat{\sigma^*_p}=\frac{\chi_p\, \hat\sigma}{im}.
\end{align}
Since $|m|\geq p$ on the support of $\chi_p$, this gives
\begin{align*}
|\widehat{\sigma^*_p}(x,m)|\leq C \, \frac{|\hat\sigma(x,m)|}{p}
\end{align*}
and \eqref{h10}(a) follows directly from this. From \eqref{h11} we also obtain $\widehat{\sigma^*_p}(x,0)=0$.

\textbf{2.) }The proof of \eqref{h10}(b) is almost the same, except that now one uses
\begin{align*}
\left|\partial_m^s\left(\frac{\chi_p}{m}\right)\right|\leq \frac{C}{p^{s+1}}.
\end{align*}
\end{proof}

\begin{prop}\label{h11a}
Let $\sigma(x,\theta)$ and $\tau(x,\theta)$ belong to $H^s_T$, $s>\frac{d+2}{2}$. Then
\begin{align}\label{h11b}
|\sigma \, \tau -(\sigma \, \tau)_p|_{H^s_T}\leq C \, |\sigma|_{H^s_T} \, |\tau|_{H^s_T} \, \sqrt{p}.
\end{align}
\end{prop}

\begin{proof}
With $*$ denoting convolution in $m$ we have
\begin{align*}
|\sigma \, \tau -(\sigma \, \tau)_p|^2_{H^s_T} &\sim \sum_{|\alpha|+k\leq s+1} |\partial_x^\alpha \, m^k \, 
(\hat\sigma *\hat\tau)(x,m) \, (1-\chi_p(m))|^2_{L^2(x,m)} \\
&\le C\, \int_{|m|\leq 2p} |(\hat\sigma *\hat\tau)(x,m)|^2_{H^s(x)} \, {\rm d}m \\
&\leq C \, \int_{|m|\leq 2p} 
\left( \int |\hat\sigma(x,m-m_1)|_{H^s(x)} \, |\hat\tau(x,m_1)|_{H^s(x)} \, {\rm d}m_1\right)^2 \, {\rm d}m \\
&\le C\, p \, |\hat\sigma(x,m)|^2_{L^2(m,H^s(x))} \, |\hat\tau(x,m)|^2_{L^2(m,H^s(x))} \leq 
C\, p \, |\sigma|^2_{H^s_T} \, |\tau|^2_{H^s_T}.
\end{align*}
\end{proof}

\begin{prop}\label{h11c}
Let $\sigma(x,\theta)$ and $\tau(x,\theta)$ belong to $\Gamma^s_T$, $s>\frac{d}{2}+3$ and let $(\sigma\tau)_p^*$ 
denote the unique primitive of $(\sigma\tau)_p$ that decays to zero as $|\theta|\to \infty$.  Then
\begin{align*}
|(\sigma \, \tau)_p^*|_{H^s_T} \leq C\, \dfrac{|\sigma|_{H^s_T} \, |\tau|_{H^s_T}}{p}.
\end{align*}
\end{prop}

\begin{proof}
Since $\Gamma^s_T$ is a Banach algebra, Proposition \ref{h9} implies $(\sigma\tau)_p^*\in\Gamma^s_T$ with 
moment zero and
\begin{align*}
|(\sigma \, \tau)_p^*|_{H^s_T} \leq C \, \frac{|(\sigma \, \tau)_p|_{H^s_T}}{p}.
\end{align*}
Since $H^s_T$ is a Banach algebra, the result now follows from Proposition \ref{h8}(a).
\end{proof}

\subsection{Proof of Theorem \ref{main}}
\label{proofmain}

\paragraph{Choice of $a$ and $\tilde a$.} The conditions on the boundary datum $G(x',\theta_0)$ are 
different in Theorems \ref{k8b} and \ref{m14ab} We need to choose $a$, $\tilde a$, and $G(x',\theta_0)$ 
so that both Theorems apply simultaneously. We also need $a$ large enough so that we can apply 
Proposition \ref{i33a} in the step \eqref{m26} of the error analysis below. These conditions are  met 
if we take $a$ and $\tilde{a}$ as in \eqref{a11z}, and choose $G\in b\Gamma^{\tilde a}_T$ such that 
$\langle G\rangle_{b\Gamma^{a+2}_T}$ is small enough. Applying Theorems \ref{k8b} and  \ref{m14ab}, 
we now have for $0<\eps\leq \eps_0$ an exact solution $U_\eps(x,\theta_0) \in E^{a-1}_T$ to the singular 
system \eqref{15p} and profiles $\cU^0\in \tilde\Gamma^{a+1}_T$, $\cU^1\in  \tilde\Gamma^{a-2}_T$ 
satisfying the profile equations \eqref{14a} and \eqref{14b}.

\paragraph{Approximation.} Fix $0<p \le 1$, we use Proposition \ref{h6} to choose moment zero approximations 
$\cU^0_p$ and $\cU^1_p$ such that
\begin{align}\label{m10}
\begin{split}
&|\cU^0-\cU^0_p|_{\cE^{a-1}_T} \le C\, \sqrt{p},\quad 
|\partial_{x_d}\cU^0-\partial_{x_d}\cU^0_p|_{\cE^{a-2}_T} \le C\, \sqrt{p},\\
&|\cU^1-\cU^1_p|_{\cE^{a-4}_T} \le C\, \sqrt{p}, \quad 
|\partial_{x_d}\cU^1-\partial_{x_d}\cU^1_p|_{\cE^{a-5}_T} \le C\, \sqrt{p}.
\end{split}
\end{align}
Having made these choices, we can now state the main result of this section, which yields the final 
convergence result of Theorem \ref{main} as an immediate corollary.

\begin{theo}\label{main2}
We make the same Assumptions as in Theorem \ref{main} and let $a$ and $\tilde a$ be chosen as in \eqref{a11z}. 
Consider the leading order approximate solution to the singular semilinear system \eqref{15p} given by
\begin{align*}
\cU^0_\eps(x,\theta_0):=\cU^0 \big( x,\theta_0,\frac{x_d}{\eps} \big),
\end{align*}
and let $U_\eps(x,\theta_0)\in E^{a-1}_T$ be the exact solution to \eqref{15p} just obtained. Then
\begin{align*}
|U_\eps(x,\theta_0)-\cU^0_\eps(x,\theta_0)|_{E^{a-6}_T}\leq C \, \eps^{\frac{1}{2M_1+3}},\quad 
\text{\rm where } M_1=\left[\frac{d}{2}+3\right].
\end{align*}
\end{theo}

\begin{proof}
We shall fill in the sketch provided in section \ref{erran}. The approximate solution to the singular system defined by 
$(\cU^0+\eps\cU^1)(x,\theta_0,\frac{x_d}{\eps})$ is too crude for the error analysis, so we must construct an additional 
corrector.

\textbf{1.) } As a first try  we could use  Proposition \ref{35} to construct $\tilde \cU^2_p(x,\theta_0,\xi_d)$ satisfying
\begin{align}\label{m13}
\tilde \cL(\partial_{\theta_0},\partial_{\xi_d})\tilde \cU^2_p =-(I-{\bf E})\left(\tilde L(\partial)\cU^1_p +D(0)\cU^1_p 
+({\rm d}D(0) \cU^0_p)\cU^0_p \right):=-(I-{\bf E})\cG.
\end{align}
This choice of $\tilde\cU^2_p(x,\theta_0,\frac{x_d}{\eps})$ turns out not to lie in any $E^r_T$ space, and is thus too 
large to be useful in the error analysis. Observe however that $\cG$ is of type $\cF$.

To remedy this problem we replace $(I-{\bf E})\cG$ by a modification $[(I-{\bf E})\cG]_{mod}$ defined as follows. 
First, using the earlier formulas for ${\bf E}\cU^1$ and $(I-{\bf E})\cU^1$ and Remark \ref{30ab}, we have
\begin{align}\label{k1}
\begin{split}
-(I-{\bf E})\cG=& \sum^3_{i=1} \left(-\sum_{k\neq i}V^i_k\tau_{k,p}-\sum_{k\neq i}X_{\phi_i}b_{k,p} 
-\sum_{k\neq i,l\neq k,l\neq i}V^i_kb_{l,p}+\sum_{k\neq i}e^i_k\tau_{k,p}\right)r_i\\
&+\sum^3_{i=1}\left(\sum_{k\neq i} c^i_k\sigma^2_{k,p}+\sum_{l\neq m}d^i_{l,m}\sigma_{l,p}\sigma_{m,p}\right)r_i 
=:A+B,
\end{split}
\end{align}
where $B$ is the second sum over $i$ in \eqref{k1}, $\sigma_{q,p}=\sigma_{q,p}(x,\theta_0+\omega_q\xi_d)$, 
and $\tau_{r,p}=\tau_{r,p}(x,\theta_0+\omega_r\xi_d)$. The problem is caused by the self-interaction terms given 
by the  sum over $k\neq i$ in $B$, so we define
\begin{align}\label{k2}
-[(I-{\bf E})\cG]_{mod}=A+\sum^3_{i=1}\left(\sum_{k\neq i} c^i_k(\sigma^2_{k,p})_p 
+\sum_{l\neq m}d^i_{l,m}\sigma_{l,p}\sigma_{m,p}\right) r_i \, ,
\end{align}
and we set
\begin{align}\label{k3}
\cU^2_p:=-{\bf R}_\infty[(I-{\bf E})\cG]_{mod}.
\end{align}
Instead of \eqref{m13} we have
\begin{align}\label{k4}
\tilde{\cL}(\partial_{\theta_0},\partial_{\xi_d})\cU^2_p=-[(I-{\bf E})\cG]_{mod}.
\end{align}
For later use we set
\begin{align*}
D(x,\theta_0,\xi_d):=(I-{\bf E})\cG-[(I-{\bf E})\cG]_{mod}
\end{align*}
and estimate
\begin{align*}
|D(x,\theta_0,\frac{x_d}{\eps})|_{E^{a}_T}\leq C \, \sqrt{p}.
\end{align*}
Indeed, using Propositions \ref{h11a} and \ref{h8}(a) we have
\begin{align*}
\Big| \left(\sigma_{k,p}^2-(\sigma_{k,p}^2)_p\right)(x,\theta_0+\omega_k\frac{x_d}{\eps}) \Big|_{E^{a}_T} 
\leq |\left(\sigma_{k,p}^2-(\sigma_{k,p}^2)_p\right)|_{H^{a+1}_T} \leq |\sigma_k|^2_{H^{a+1}_T} \, \sqrt{p}.
\end{align*}

We then define the corrected approximate solution
\begin{align*}
\cU_\eps(x,\theta_0) := \cU^0(x,\theta_0,\frac{x_d}{\eps})+\eps\cU^1(x,\theta_0,\frac{x_d}{\eps})
+\eps^2 \, \cU^2_p(x,\theta_0,\frac{x_d}{\eps}).
\end{align*}
Since $\cU^1\in \tilde \Gamma^{a-2}_T$, Proposition \ref{h1} implies $\cU^1_\eps\in E^{a-3}_T$.

\textbf{2.) Estimate of $|\cU^2_{p,\eps}|_{E^{a-4}_T}$.}  For $A$ as in \eqref{k1} let us write $A =\sum^3_{i=1} 
a_i(x,\theta_0,\xi_d)r_i$. By \eqref{k3}, \eqref{k2} and the formula \eqref{R} for ${\bf R}_\infty$, for $i=1,2,3$ we 
must estimate $|b_i(x,\theta_0,\frac{x_d}{\eps})|_{E^{s-2}_T}$, where $b_i(x,\theta_0,\xi_d)$ reads as follows:
\begin{align}\label{k8}
&\sum_{k\neq i} c^i_k \int_{\xi_d}^{+\infty} \left(\sigma^2_{k,p}\right)_p(x,\theta_0+\omega_i\xi_d+s(\omega_k-\omega_i)) 
\, {\rm d}s \notag\\
+&\sum_{m\neq i}d^i_{i,m} \int_{\xi_d}^{+\infty} \sigma_{i,p}(x,\theta_0+\omega_i\xi_d)\sigma_{m,p} 
(x,\theta_0+\omega_i\xi_d +s(\omega_m-\omega_i)) \, {\rm d}s \notag\\
+&\sum_{l\neq i} d^i_{l,i} \int_{\xi_d}^{+\infty} \sigma_{l,p}(x,\theta_0+\omega_i\xi_d+s(\omega_l-\omega_i)) 
\sigma_{i,p}(x,\theta_0+\omega_i\xi_d) \, {\rm d}s\\
+&\sum_{l\neq m,l\neq i,m\neq i}d^i_{l,m} \int_{\xi_d}^{+\infty} \sigma_{l,p}(x,\theta_0+\omega_i\xi_d+s(\omega_l-\omega_i)) 
\sigma_{m,p}(x,\theta_0+\omega_i\xi_d+s(\omega_m-\omega_i)) \, {\rm d}s \notag\\
+&\int_{\xi_d}^{+\infty} a_i(x,\theta_0+\omega_i(\xi_d-s),s) \, {\rm d}s=\sum^5_{r=1}b_{i,r}(x,\theta_0,\xi_d),\notag
\end{align}
where $b_{i,r}$, $r=1,\dots,5$ are defined by the respective lines of \eqref{k8}. Since $\cU^0\in \tilde\Gamma^{a+1}_T$, 
using Corollary \ref{h17} we find
\begin{align*}
\Big| b_{i,1} \Big( x,\theta_0,\frac{x_d}{\eps} \Big) \Big|_{E^{a}_T}\leq C \sum_{k\neq i}\frac{|\sigma_k|^2_{H^{a+1}_T}}{p} 
\leq C/p.
\end{align*}
Similarly, from Proposition \ref{h19} we get
\begin{align*}
\Big| b_{i,2} \Big( x,\theta_0,\frac{x_d}{\eps} \Big) \Big|_{E^{a}_T} \leq C/p;\;\; 
\Big| b_{i,3} \Big( x,\theta_0,\frac{x_d}{\eps} \Big) \Big|_{E^{a}_T} \leq C/p.
\end{align*}
Proposition \ref{h12} on transversal interactions implies
\begin{align*}
\Big| b_{i,4} \Big( x,\theta_0,\frac{x_d}{\eps} \Big) \Big|_{E^{a}_T} \leq  \frac{C}{p^{M_1}},
\end{align*}
where we have used Proposition \ref{h8}(b) to estimate $|\sigma_{m,p}|_{\Gamma^{M_1}_T}\leq 
\frac{C}{p^{M_1}} \, |\sigma_{m}|_{\Gamma^{M_1}_T}$.

Since $A\in\tilde\Gamma^{a-3}_T$, Proposition \ref{h20} implies $|b_{i,5}(x,\theta_0,\frac{x_d}{\eps})|_{E^{a-4}_T} 
\leq  C/p$, so adding up we obtain
\begin{align*}
\big| \cU^2_{p,\eps} \big|_{E^{a-4}_T}\leq \frac{C}{p^{t}}.
\end{align*}
To estimate $(\partial_{x_d}\cU^2_p)_\eps$ we differentiate \eqref{k8} and estimate as above to find
\begin{align*}
\big| (\partial_{x_d}\cU^2_p)_\eps \big|_{E^{a-5}_T}\leq \frac{C}{p^{t+1}}.
\end{align*}

\textbf{3.) } Next we study the error upon substituting the corrected approximate solution $\cU_\eps$ into the singular 
system \eqref{15p}. Letting $\bL_\eps$ denote the operator defined by the left side of \eqref{15p}, we compute
\begin{equation}\label{m15}
\bL_\eps(\cU_\eps) =\eps \, \left[ (\tilde \cL(\partial_{\theta_0},\partial_{\xi_d})\cU^2_p)|_{\xi_d=\frac{x_d}{\eps}}
+\left(\tilde L(\partial)\cU^1+D(0)\cU^1+({\rm d}D(0) \cU^0)\cU^0\right)|_{\xi_d=\frac{x_d}{\eps}}
\right] +O(\eps^2).
\end{equation}
Here the profile equations \eqref{13} imply that the terms of order $\eps^{-1}$ and $\eps^0$ vanish. 
Using \eqref{k4} we can rewrite the  coefficient of $\eps$ in \eqref{m15} as
\begin{multline*}
\left[ \tilde L(\partial)(\cU^1-\cU^1_p)+D(0)(\cU^1-\cU^1_p) +({\rm d}D(0) \cU^0)\cU^0
-({\rm d}D(0) \cU^0_p)\cU^0_p \right]
|_{\xi_d=\frac{x_d}{\eps}} \\
+\left[ {\bf E}\left(\tilde L(\partial)\cU^1_p+D(0)\cU^1_p +({\rm d}D(0) \cU^0_p)\cU^0_p \right) \right]
|_{\xi_d=\frac{x_d}{\eps}} +D(x,\theta_0,\frac{x_d}{\eps}):=A+B+ D(x,\theta_0,\frac{x_d}{\eps})\, .
\end{multline*}
Using \eqref{m10}, Proposition \ref{h1}, and the fact that $E^s_T$ is a Banach algebra for $s\geq [(d+1)/2]$,
we see that
\begin{align*}
|A|_{E^{a-5}(\Omega_T)}< C\sqrt{p}.
\end{align*}

To estimate $B$, let
\begin{align*}
F=\tilde L(\partial)\cU^1 +D(0)\cU^1 +({\rm d}D(0) \cU^0)\cU^0 \text{ and }F_p
=\tilde L(\partial)\cU^1_p +D(0)\cU^1_p +({\rm d}D(0) \cU^0_p)\cU^0_p.
\end{align*}
The profile equation \eqref{14b}(b) implies ${\bf E}F=0$. Since $F$ and $F_p$ are of type $\cF$, we can use the 
explicit formula for the action of ${\bf E}$ in Remark \ref{30ab} and \eqref{m10} to obtain
\begin{align*}
|B|_{E^{a-5}_T}=\left|({\bf E}F_p)|_{\xi_d=\frac{x_d}{\eps}}\right|_{E^{a-5}_T}
=\left|({\bf E}(F-F_p))|_{\xi_d=\frac{x_d}{\eps}}\right|_{E^{a-5}_T} \le C\, \sqrt{p}.
\end{align*}

\textbf{4.) } The $O(\eps^2)$ terms in \eqref{m15} consist of
\begin{align}\label{m18}
\left| \eps^2 \, \left(L(\partial)\cU^2_p(x,\theta_0,\xi_d)\right)|_{\xi_d=\frac{x_d}{\eps}}
\right|_{E^{a-5}_T}\leq C\, \frac{\eps^2}{p^{t+1}},
\end{align}
as well as terms coming from the Taylor expansion of $D(\eps\cU_\eps)\cU_\eps$ like $\eps^2 \, ({\rm d}D(0)
\cU^0)\cU^1 |_{\xi_d=\frac{x_d}{\eps}}$, all of which satisfy better estimates than \eqref{m18}.

Setting $R_\eps(x,\theta_0):=\bL_\eps(\cU_\eps)$, we have shown
\begin{align}\label{m19}
|R_\eps|_{E^{a-5}_T} \leq C \, \eps \, \left( \sqrt{p}+\frac{\eps}{p^{t+1}} \right).
\end{align}

\textbf{5.) } The boundary profile equations \eqref{14}, and the fact that the traces of $\cU^0_\eps$ and $\cU^1_\eps$ 
lie in $H^{a}(b\Omega_T)$ and $H^{a-3}(b\Omega_T)$, respectively, imply
\begin{align}\label{m20}
\left\langle r_\eps(x',\theta_0)\right\rangle_{H^{a-4}(b\Omega_T)} \leq C\frac{\eps^2}{p^t},\quad 
\text{\rm where } r_\eps :=\psi(\eps\cU_\eps) \, \cU_\eps -\eps \, G(x',\theta_0).
\end{align}
Indeed, these $O(\eps^2)$ terms include
\begin{align*}
\left\langle \eps^2 \, B \, \cU^2_{p,\eps}(x',0,\theta_0)\right\rangle_{H^{a-4}(b\Omega_T)} \leq C\, \frac{\eps^2}{p^t} \, ,
\end{align*}
and other terms satisfying the same estimate coming from the Taylor expansion of $\psi(\eps\cU_\eps)\, \cU_\eps$.

\textbf{6.) } Next we consider the singular problem satisfied by the difference $W_\eps:=U_\eps-\cU_\eps$:
\begin{align}\label{m22}
\begin{split}
&\partial_d W_\eps +\bA\left(\partial_{x'}+\frac{\beta\partial_{\theta_0}}{\eps}\right)W_\eps
+D_2(\eps U_\eps,\eps \cU_\eps)W_\eps=-R_\eps\\
&\psi_2(\eps U_\eps,\eps \cU_\eps)W_\eps=-r_\eps\text{ on }x_d=0\\
&W_\eps=0\text{ in }t<0,
\end{split}
\end{align}
where
\begin{align*}
D_2(\eps U_\eps,\eps \cU_\eps)W_\eps &:= D(\eps U_\eps)U_\eps-D(\eps \cU_\eps)\cU_\eps \\
&=D(\eps U_\eps)W_\eps +\left(\int^1_0 {\rm d}D(\eps\cU_\eps+s\eps (U_\eps-\cU_\eps)) \, W_\eps \, {\rm d}s\right) 
\, (\eps\cU_\eps) \, ,
\end{align*}
and $\psi_2(\eps U_\eps,\eps \cU_\eps)W_\eps$ is defined similarly.   Since $U_\eps\in E^{a-1}_T$ and
$\cU_\eps\in E^{a-4}_T$, a short computation shows
\begin{align*}
\psi_2(\eps U_\eps,\eps \cU_\eps)W_\eps =\psi(\eps U_\eps) W_\eps +({\rm d}\psi(\eps U) W_\eps) \eps\cU_\eps 
+O\left( \frac{\eps^2}{p^{M_1}} \right) =\cB(\eps U,\eps \cU)W_\eps +O\left( \frac{\eps^2}{p^{M_1}} \right),
\end{align*}
where the error term is measured in $H^{a-4}(b\Omega_T)$ and $\cB$ is defined in \eqref{ca2}. Similarly,
\begin{align*}
D_2(\eps U_\eps,\eps \cU_\eps)W_\eps =\cD(\eps U,\eps \cU)W_\eps+O\left( \frac{\eps^2}{p^{M_1}} \right) 
\text{ in }E^{a-4}_T.
\end{align*}
Thus, using \eqref{m19} and \eqref{m20} we find
\begin{align*}
\begin{split}
&\partial_d W_\eps+\bA\left(\partial_{x'}+\frac{\beta\partial_{\theta_0}}{\eps}\right)W_\eps
+\cD(\eps U_\eps,\eps \cU_\eps)W_\eps=\eps \, O\left( \sqrt{p}+\frac{\eps}{p^{M_1+1}} \right) \text{ in }E^{a-5}_T\\
&\cB(\eps U_\eps,\eps \cU_\eps)W_\eps|_{x_d=0} =O\left( \frac{\eps^2}{p^{M_1}} \right) \text{ in }H^{a-4}(b\Omega_T)\\
&W_\eps=0\text{ in }t<0.
\end{split}
\end{align*}
Provided $\langle G\rangle_{b\Gamma^{\alpha+2}_T}$ is small enough, we can apply the estimate of Proposition 
\ref{i33a} to obtain
\begin{align}\label{m26}
|W_\eps|_{E^{a-6}_T} \leq C\, \left( \sqrt{p}+\frac{\eps}{p^{M_1+1}} \right).
\end{align}
Setting $p=\eps^b$ and $\sqrt{p}=\eps/p^{M_1+1}$, we find $b=\frac{2}{2M_1+3}$. Thus, $|W_\eps|_{E^{a-6}_T} \leq 
C\, \eps^{\frac{1}{2M_1+3}}$, so \eqref{m26} implies
\begin{align*}
|U_\eps-\cU^0_\eps|_{E^{a-6}_T} \leq C \, \eps^{\frac{1}{2M_1+3}}.
\end{align*}
\end{proof}


\section{Extension to the general $N\times N$ system}
\label{extension}

\emph{\quad} It is only the discussion of profiles and the error analysis that needs to be extended to the general case. 
For each $m\in \{1,\dots,M\}$ we let
\begin{equation*}
\ell_{m,k},\;k=1,\dots,\nu_{k_m}
\end{equation*}
denote a basis of real vectors for the left eigenspace of the real matrix $i\cA(\beta)$ associated to the real eigenvalue 
$-\omega_m$ and chosen to satisfy
\begin{equation*}
\ell_{m,k}\cdot r_{m',k'}=\begin{cases}1, \;\text{ if }m=m'\text{ and }k=k',\\
0,\;\text{ otherwise.}\end{cases}
\end{equation*}
For $v\in\C^N$ we set
\begin{equation*}
P_{m,k}v :=(\ell_{m,k}\cdot v) \, r_{m,k} \, ,\quad \text{\rm (no complex conjugation here)}.
\end{equation*}

Functions of type $\cF$ (see Definition \ref{30}) have the form
\begin{equation*}
F(x,\theta_0,\xi_d)=\sum_{m=1}^M\sum_{k=1}^{\nu_{k_m}}F_{m,k}(x,\theta_0,\xi_d) \, r_{m,k} \, ,
\end{equation*}
where
\begin{multline*}
F_{m,k}=\sum_{m'} f^{m,k}_{m'}(x,\theta_0+\omega_{m'} \, \xi_d) \\
+\sum_{m',k',m'',k''} g^{m,k}_{m',k',m'',k''}(x,\theta_0+\omega_{m'} \, \xi_d) \, 
h^{m,k}_{m',k',m'',k''}(x,\theta_0+\omega_{m''} \, \xi_d) \, .
\end{multline*}
The operator ${\bf E}$ is given by
\begin{equation*}
{\bf E}F:=\sum_{m,k} \left( \lim_{T\to\infty} \dfrac{1}{T} \int^T_0 F_{m,k}(x,\theta_0+\omega_m(\xi_d-s),s)\, 
{\rm d}s \right) \, r_{m,k} \, ,
\end{equation*}
and we leave it to the reader to formulate the analogue of Remark \ref{30ab}. On functions of type $\cF$ such 
that ${\bf E}F=0$ the action of the operator ${\bf R}_\infty$ is given by
\begin{equation*}
{\bf R}_\infty F:=-\sum_{m,k}\left(\int_{\xi_d}^{+\infty} F_{m,k}(x,\theta_0+\omega_m(\xi_d-s),s)\, {\rm d}s \right) \, 
r_{m,k} \, .
\end{equation*}
The general form of the profile equations \eqref{14a}, \eqref{14b} still applies. With
\begin{equation*}
W(x,\theta_0,\xi_d)=\sum_{m,k}w_{m,k}(x,\theta_0,\xi_d) \, r_{m,k},
\end{equation*}
the decomposition of Proposition \ref{23a} now has the form
\begin{align}\label{r7}
\tilde L(\partial)W=\sum_{m,k}(X_{\phi_m}\, w_{m,k}) \, r_{m,k} 
+\sum_{m,k} \left( \sum_{m'\neq m,k'} V^{m,k}_{m',k'} \, w_{m',k'} \right) \, r_{m,k},
\end{align}
where $V^{m,k}_{m',k'}$ is the tangential vector field
\begin{equation*}
V^{m,k}_{m',k'}=\sum^{d-1}_{j=0}(\ell_{m,k}A_j r_{m',k'})\partial_{x_j}.
\end{equation*}

Recalling  that $\{1,\dots,M\}=\cO\cup\cI$, where $\cO$ and $\cI$ contain the indices corresponding to outgoing 
and incoming phases, we have the expressions in \eqref{a11} for $\cU^0$ and $\cU^1$ in the general case. As 
in \eqref{defe} we have
\begin{equation*}
\ker B \cap \E^s(\beta)= \mathrm{Span }\{ e \} =\sum _{m\in\cI}e_m 
=\sum_{m\in \cI} \sum_{k=1}^{\nu_{k_m}}e_{m,k} \, ,
\end{equation*}
where $e_{m,k}:=P_{m,k}e$. Using \eqref{r7} and arguing as in Propositions \ref{m3az} and \ref{m5}, we find that 
the key subsystem \eqref{m6} now takes the form
\begin{align}\label{r8}
\begin{split}
&X_{\phi_m}\sigma_{m,k} +e_{m,k} \, \sigma_{m,k}=0,\; m\in \cI,\;k=1,\dots,\nu_{k_m} \, ,\\
&X_{Lop}a +c\, a +d\, \partial_{\theta_0}(a^2) =-b\cdot \partial_{\theta_0}G \text{ on }\{x_d=0,\xi_d=0\} \, ,
\end{split}
\end{align}
where
\begin{equation*}
\sigma_{m,k}(x',0,\theta_0) \, r_{m,k} =a(x',\theta_0) \, e_{m,k} \, ,
\end{equation*}
and all unknowns are zero in $t<0$. The key subsystem and then the full profile system are solved as before 
(using a Picard iteration for $a$ and integration along the characteristics for the $\sigma_{m,k}$'s), and Theorem 
\ref{m14ab} on the existence and regularity of $\cU^0$ and $\cU^1$ still holds as stated.

The error analysis in the proof of Theorem \ref{main2} goes through with the obvious minor changes. For example, 
the troublesome self-interaction terms $c^i_k\sigma^2_{k,p}$ in the sum that is part of $B$ in \eqref{k1} are now 
replaced by terms of the form $c^i_{m,k,k'} \sigma_{m,k,p}\sigma_{m,k',p}$, $m\neq i$, where the index $p$ as 
before denotes a moment-zero approximation. These terms are handled just as before by introducing $-[(I-{\bf E}) 
\cG]_{mod}$, see \eqref{k2}, in which they are replaced by  $c^i_{m,k,k'} (\sigma_{m,k,p}\sigma_{m,k',p})_p$. 
The contribution of these terms to $\cU^2_p$ is estimated as before using Corollary \ref{h17}. The statement 
of Theorem \ref{main2} applies unchanged in the general case.

\appendix
\section{Estimates of interaction integrals}
\label{interaction}

\emph{\quad}Pulses do not interact to produce resonances that affect the leading order profiles as in the 
wavetrain case. However, interaction integrals must be estimated carefully in the error analysis of Section 
\ref{error}. All results below are proved in our earlier paper \cite{CW4} so we shall not reproduce the proofs 
here.

The following propositions are used in the error analysis for estimating integrals appearing in the definition 
of $\cU^2_{p,\eps}$ (step \textbf{2} of Theorem \ref{main2}). In particular, we must estimate primitives of 
products of pulses. In some of the estimates below we must introduce moment-zero approximations to 
avoid errors that would fail to lie in any $E^r_T$ space, which are thus too large to be useful in the error 
analysis. We begin with an estimate of ``transversal interactions", see Proposition 4.10 in \cite{CW4}.

\begin{prop}\label{h12}
Let $t$ be the smallest integer greater than $\frac{d}{2}+3$, and let $s\geq 0$. Let $\sigma_1(x,\theta)$, 
$\sigma_2(x,\theta)$ belong to $\Gamma^t_T\cap H^{s+1}_T$, and define
\begin{align*}
u(x,\theta_0,\xi_d):= -\int^{+\infty}_{\xi_d} \sigma_1(x,\theta_0+\omega\xi_d+\alpha s) \, 
\sigma_2(x,\theta_0+\omega\xi_d+s) \, {\rm d}s,
\end{align*}
where $\omega$, $\alpha$ are real and $\alpha\notin\{0,1\}$. With $u_\eps(x,\theta_0):= 
u(x,\theta_0,\frac{x_d}{\eps})$, we have
\begin{align*}
|u_\eps|_{E^s_T}\leq C \, \big( |\sigma_1|_{H^{s+1}_T} \, |\sigma_2|_{\Gamma^t_T} 
+|\sigma_2|_{H^{s+1}_T} \, |\sigma_1|_{\Gamma^t_T} \big) \, ,
\end{align*}
uniformly for $\eps\in (0,1]$.
\end{prop}

The previous estimate of transversal interactions did not require the use of moment-zero approximations. 
However, nontransversal interactions of pulses can produce errors that do not lie in any $E^r_T$ space, and 
are thus too big to be helpful in the error analysis. Thus, we are forced to use a moment-zero approximation in 
the next two propositions (see corresponding Proposition 4.11, Corollary 4.12 and Proposition 4.14 in \cite{CW4}).

\begin{prop}\label{h16}
Let $\sigma(x,\theta)$ and $\tau(x,\theta)$ belong to $\Gamma^s_T$, $s>\frac{d}{2}+3$. For $\alpha, \omega\in\R$, 
$\alpha\neq 0$ set
\begin{equation*}
f(x,\theta_0,\xi_d) :=-\int_{\xi_d}^{+\infty} (\sigma \, \tau)_p(x,\theta_0+\omega \, \xi_d+\alpha \, s) \, {\rm d}s.
\end{equation*}
Then
\begin{equation*}
\left| f \Big( x,\theta_0,\frac{x_d}{\eps} \Big) \right|_{E^{s-1}_T} \leq C \, \dfrac{|\sigma|_{H^s_T} \, |\tau|_{H^s_T}}{p} \, .
\end{equation*}
\end{prop}


\begin{cor}\label{h17}
Let $\sigma(x,\theta)$, $\tau(x,\theta)$, and $\omega,\alpha$ be as in Proposition \ref{h16} and set
\begin{align*}
g(x,\theta_0,\xi_d) :=-\int_{\xi_d}^{+\infty} (\sigma_p \, \tau_p)_p(x,\theta_0+\omega \, \xi_d+\alpha \, s) \, {\rm d}s.
\end{align*}
Then
\begin{equation*}
\left| g \Big( x,\theta_0,\frac{x_d}{\eps} \Big) \right|_{E^{s-1}_T} \leq C \, \dfrac{|\sigma|_{H^s_T} \, |\tau|_{H^s_T}}{p} \, .
\end{equation*}
\end{cor}




\begin{prop}\label{h19}
For $s> \frac{d}{2}+3$ let $\sigma(x,\theta)\in H^s_T$, $\tau(x,\theta)\in\Gamma^s_T$. With $\omega,\alpha \in \R$, 
$\alpha\neq 0$ set
\begin{align*}
j(x,\theta_0,\xi_d) :=-\sigma (x,\theta_0+\omega \, \xi_d) \, \int_{\xi_d}^{+\infty} 
\tau_p (x,\theta_0+\omega \, \xi_d+\alpha \, s) \, {\rm d}s.
\end{align*}
Then 
\begin{equation*}
\left| j(x,\theta_0,\frac{x_d}{\eps}) \right|_{E^{s-2}_T} \leq C \, \dfrac{|\sigma|_{H^s_T} \, |\tau|_{H^{s-1}_T}}{p} \, .
\end{equation*}
\end{prop}


The next  Proposition requires moment zero approximations because forcing by a noninteracting pulse can 
also produce errors that fail to lie in any $E^r_T$ space, see Proposition 4.15 in \cite{CW4}.

\begin{prop}\label{h20}
For $s>\frac{d}{2}+3$ and $\omega,\alpha\in\R$, $\alpha\neq 0$, let $\sigma\in\Gamma^s_T$ and set
\begin{align*}
k(x,\theta_0,\xi_d)=-\int_{\xi_d}^{+\infty} \sigma_p(x,\theta_0+\omega\xi_d+\alpha s) \, {\rm d}s.
\end{align*}
Then
\begin{equation*}
\left| k(x,\theta_0,\frac{x_d}{\eps}) \right|_{E^{s-1}_T} \leq C \, \dfrac{|\sigma|_{H^s_T}}{p} \, .
\end{equation*}
\end{prop}

\section{Singular pseudodifferential calculus for pulses}
\label{calculus}

\emph{\quad} Here we summarize the parts of the singular pulse calculus constructed in \cite{CGW2} that 
are needed in this article. First we define the singular Sobolev spaces used to describe mapping properties.

The variable in $\R^{d+1}$ is denoted $(x,\theta)$, $x \in \R^d$, $\theta \in \R$, and the associated frequency 
is denoted $(\xi,k)$. In this new context, the singular Sobolev spaces are defined as follows. We consider a 
fixed vector $\beta \in \R^d \setminus \{ 0\}$. Then for $s \in \R$ and $\eps \in \, ]0,1]$, the anisotropic Sobolev 
space $H^{s,\eps} (\R^{d+1})$ is defined by
\begin{multline*}
H^{s,\eps}(\R^{d+1}) := \Big\{ u \in {\mathcal S}'(\R^{d+1}) \, / \, \widehat{u} \in L^2_{\rm loc}(\R^{d+1}) \\
\text{\rm and} \quad \int_{\R^{d+1}} \left( 1+\left| \xi+\dfrac{k \, \beta}{\eps} \right|^2 \right)^s
\, \big| \widehat{u}(\xi,k) \big|^2 \, {\rm d}\xi \, {\rm d}k <+\infty \Big\} \, .
\end{multline*}
Here $\widehat{u}$ denotes the Fourier transform of $u$ on $\R^{d+1}$. The space $H^{s,\eps}(\R^{d+1})$ is
equipped with the family of norms
\begin{equation*}
\forall \, \gamma \ge 1 \, ,\quad \forall \, u \in H^{s,\eps}(\R^{d+1}) \, ,\quad
\| u \|_{H^{s,\eps},\gamma}^2 := \dfrac{1}{(2\, \pi)^{d+1}} \, \int_{\R^{d+1}}
\left( \gamma^2 +\left| \xi+\dfrac{k \, \beta}{\eps} \right|^2 \right)^s
\, \big| \widehat{u}(\xi,k) \big|^2 \, {\rm d}\xi \, {\rm d}k \, .
\end{equation*}
When $m$ is an integer, the space $H^{m,\eps} (\R^{d+1})$ coincides with the space of functions $u \in L^2
(\R^{d+1})$ such that the derivatives, in the sense of distributions,
\begin{equation*}
\left( \partial_{x_1} +\dfrac{\beta_1}{\eps} \, \partial_\theta \right)^{\alpha_1} \dots
\left( \partial_{x_d} +\dfrac{\beta_d}{\eps} \, \partial_\theta \right)^{\alpha_d} \, u \, ,\quad
\alpha_1+\dots+\alpha_d \le m \, ,
\end{equation*}
belong to $L^2 (\R^{d+1})$. In the definition of the norm $\| \cdot \|_{H^{m,\eps},\gamma}$, one power of
$\gamma$ counts as much as one derivative.

\subsection{Symbols}

\emph{\quad} Our singular symbols are built from the following sets of classical symbols.

\begin{defn}\label{n1}
Let $\cO\subset \R^N$ be an open subset that contains the origin.  For $m\in\R$ we let $\bfS^m(\cO)$ denote
the class of all functions $\sigma:\cO\times \R^d\times [1,\infty)\to \C^{N \times N}$, $N \ge 1$, such that 
$\sigma$ is $C^\infty$ on $\cO \times \R^d$ and for all compact sets $K\subset \cO$:
\begin{equation*}
\sup_{v\in K} \, \sup_{\xi'\in\R^d} \, \sup_{\gamma\geq 1} \, (\gamma^2+|\xi|^2)^{-(m-|\nu|)/2} \,
|\partial^\alpha_v\partial_{\xi'}^\nu \sigma(v,\xi,\gamma)| \leq C_{\alpha,\nu,K}.
\end{equation*}
\end{defn}

Let ${\mathcal C}^k_b(\R^{d+1})$, $k \in \N$, denote the space of continuous and bounded functions 
on $\R^{d+1}$, whose derivatives up to order $k$ are continuous and bounded. Let us first define the 
singular symbols.

\begin{defn}[Singular symbols]
\label{def4}
Fix $\beta\in\R^d\setminus 0$, let $m \in \R$, and let $n \in \N$. Then we let $S^m_n$ denote the set of 
families of functions $(a_{\eps,\gamma})_{\eps \in ]0,1],\gamma \ge 1}$ that are constructed as follows:
\begin{equation}
\label{singularsymbolp}
\forall \, (x,\theta,\xi,k) \in \R^{d+1} \times \R^{d+1} \, ,\quad a_{\eps,\gamma} (x,\theta,\xi,k) =
\sigma \left( \eps \, V(x,\theta),\xi+\dfrac{k \, \beta}{\eps},\gamma \right) \, ,
\end{equation}
where $\sigma \in {\bf S}^m({\mathcal O})$, $ V$ belongs to the space ${\mathcal C}^n_b (\R^{d+1})$ and 
where furthermore $V$ takes its values in a convex compact subset $K$ of ${\mathcal O}$ that contains 
the origin (for instance $K$ can be a closed ball centered round the origin).
\end{defn}

All results below extend to the case where in place of a function $V$ that is independent of $\eps$, the 
representation \eqref{singularsymbolp} is considered with a function $V_\eps$ that is indexed by $\eps$, 
provided that we assume that all functions $\eps \, V_\eps$ take values in a {\it fixed} convex compact 
subset $K$ of ${\mathcal O}$ that contains the origin, and $(V_\eps)_{\eps \in (0,1]}$ is a bounded family 
of ${\mathcal C}^n_b (\R^{d+1})$. Such singular symbols with a function $V_\eps$ are exactly the kind 
of symbols that we manipulated in the construction of exact solutions to the singular system \eqref{15p}.

\subsection{Definition of operators and action on Sobolev spaces}
\label{sect8}

To each symbol $a = (a_{\eps,\gamma})_{\eps \in ]0,1],\gamma \ge 1} \in S^m_n$ given by the formula 
\eqref{singularsymbolp}, we associate a singular pseudodifferential operator $\opeg (a)$, with $\eps \in \, ]0,1]$ 
and $\gamma \ge 1$, whose action on a function $u \in {\mathcal S} (\R^{d+1} ; \C^N)$ is defined by
\begin{equation}
\label{singularpseudop}
\opeg (a) \, u \, (x,\theta) := \dfrac{1}{(2\, \pi)^{d+1}} \, \int_{\R^{d+1}} {\rm e}^{i\, (\xi \cdot x +k \, \theta)} \,
\sigma \left( \eps \, V(x,\theta),\xi+\dfrac{k \, \beta}{\eps},\gamma \right) \, \widehat{u} (\xi,k)
\, {\rm d}\xi \, {\rm d}k \, .
\end{equation}
Let us briefly note that for the Fourier multiplier $\sigma (v,\xi,\gamma) =i\, \xi_1$, the corresponding
singular operator is $\partial_{x_1} +(\beta_1/\eps) \, \partial_\theta$. We now describe the action of 
singular pseudodifferential operators on Sobolev spaces.

\begin{prop}
\label{prop13}
Let $n \ge d+1$, and let $a \in S^m_n$ with $m \le 0$. Then $\opeg (a)$ in \eqref{singularpseudop} defines
a bounded operator on $L^2 (\R^{d+1})$: there exists a constant $C>0$, that only depends on $\sigma$
and $V$ in the representation \eqref{singularsymbolp}, such that for all $\eps \in \, ]0,1]$ and for all
$\gamma \ge 1$, there holds
\begin{equation*}
\forall \, u \in {\mathcal S} (\R^{d+1}) \, ,\quad \left\| \opeg (a) \, u \right\|_0 \le \dfrac{C}{\gamma^{|m|}} \, \| u \|_0 \, .
\end{equation*}
\end{prop}

The constant $C$ in Proposition \ref{prop13} depends uniformly on the compact set in which $V$ takes its
values and on the norm of $V$ in ${\mathcal C}^{d+1}_b$. For operators defined by symbols of order $m>0$ 
we have:

\begin{prop}
\label{prop14}
Let $n \ge d+1$, and let $a \in S^m_n$ with $m>0$. Then $\opeg (a)$ in \eqref{singularpseudop} defines
a bounded operator from $H^{m,\eps}(\R^{d+1})$ to $L^2 (\R^{d+1})$: there exists a constant $C>0$, that
only depends on $\sigma$ and $V$ in the representation \eqref{singularsymbolp}, such that for all $\eps \in
\, ]0,1]$ and for all $\gamma \ge 1$, there holds
\begin{equation*}
\forall \, u \in {\mathcal S} (\R^{d+1}) \, ,\quad \left\| \opeg (a) \, u \right\|_0 \le C \, \| u \|_{H^{m,\eps},\gamma} \, .
\end{equation*}
\end{prop}

The next proposition describes the smoothing effect of operators of order $-1$.

\begin{prop}
\label{prop15}
Let $n \ge d+2$, and let $a \in S^{-1}_n$. Then $\opeg (a)$ in \eqref{singularpseudop} defines a bounded
operator from $L^2 (\R^{d+1})$ to $H^{1,\eps}(\R^{d+1})$: there exists a constant $C>0$, that only depends
on $\sigma$ and $V$ in the representation \eqref{singularsymbolp}, such that for all $\eps \in \, ]0,1]$ and for
all $\gamma \ge 1$, there holds
\begin{equation*}
\forall \, u \in {\mathcal S} (\R^{d+1}) \, ,\quad \left\| \opeg (a) \, u \right\|_{H^{1,\eps},\gamma} \le C \, \| u \|_0 \, .
\end{equation*}
\end{prop}

\begin{rem}\label{a4}
\textup{In applications of the pulse calculus, we verify the hypothesis that for $V$ as in \eqref{singularsymbolp}, 
$V\in \mathcal{C}^n_b(\R^{d+1})$, by showing $V\in H^s(\R^{d+1})$ for some $s>\frac{d+1}{2}+n$.}
\end{rem}

\subsection{Adjoints and products}
\label{sect9}

For proofs of the following results we refer to \cite{CGW2}. The two first results deal with adjoints of singular 
pseudodifferential operators while the last two deal with products.

\begin{prop}
\label{prop18}
Let $a=\sigma(\eps V,\X,\gamma) \in S_n^0$, $n \ge 2\, (d+1)$, where $V\in H^{s_0}(\R^{d+1})$ for some 
$s_0>\frac{d+1}{2}+1$, and let $a^*$ denote the conjugate transpose of the symbol $a$. Then $\opeg (a)$ 
and $\opeg (a^*)$ act boundedly on $L^2$ and there exists a constant $C \ge 0$ such that for all $\eps \in 
\, ]0,1]$ and for all $\gamma \ge 1$, there holds
\begin{equation*}
\forall \, u \in {\mathcal S} (\R^{d+1}) \, ,\quad
\left\| \opeg (a)^* \, u -\opeg (a^*) \, u \right\|_0 \le \dfrac{C}{\gamma} \, \| u \|_0 \, .
\end{equation*}

If $n \ge 3\, d +3$, then for another constant $C$, there holds
\begin{equation*}
\forall \, u \in {\mathcal S} (\R^{d+1}) \, ,\quad
\left\| \opeg (a)^* \, u -\opeg (a^*) \, u \right\|_{H^{1,\eps},\gamma} \le C \, \| u \|_0 \, ,
\end{equation*}
uniformly in $\eps$ and $\gamma$.
\end{prop}

\begin{prop}
\label{prop19}
Let $a=\sigma(\eps V,\X,\gamma) \in S_n^1$, $n \ge 3\, d +4$, where $V\in H^{s_0}(\R^{d+1})$ for some 
$s_0>\frac{d+1}{2}+1$, and let $a^*$ denote the conjugate transpose of the symbol $a$. Then $\opeg (a)$ 
and $\opeg (a^*)$ map $H^{1,\eps}$ into $L^2$ and there exists a family of operators $R^{\eps,\gamma}$ 
that satisfies
\begin{itemize}
 \item there exists a constant $C \ge 0$ such that for all $\eps \in \, ]0,1]$ and for all
      $\gamma \ge 1$, there holds
\begin{equation*}
\forall \, u \in {\mathcal S} (\R^{d+1}) \, ,\quad
\left\| R^{\eps,\gamma} \, u \right\|_0 \le C \, \| u \|_0 \, ,
\end{equation*}

 \item the following duality property holds
\begin{equation*}
\forall \, u,v \in {\mathcal S} (\R^{d+1}) \, ,\quad
\langle \opeg (a) \, u,v \rangle_{L^2} -\langle u,\opeg (a^*) \, v \rangle_{L^2} =\langle
R^{\eps,\gamma} \, u,v \rangle_{L^2} \, .
\end{equation*}
In particular, the adjoint $\opeg (a)^*$ for the $L^2$ scalar product maps $H^{1,\eps}$ into $L^2$.
\end{itemize}
\end{prop}

\begin{prop}
\label{prop20}
(a)\; Let $a,b \in S_n^0$, $n \ge 2\, (d+1)$, and suppose $b=\sigma(\eps V,\X,\gamma)$ where $V \in 
H^{s_0}(\R^{d+1})$ for some $s_0>\frac{d+1}{2}+1$. Then there exists a constant $C \ge 0$ such that 
for all $\eps \in \, ]0,1]$ and for all $\gamma \ge 1$, there holds
\begin{equation*}
\forall \, u \in {\mathcal S} (\R^{d+1}) \, ,\quad
\left\| \opeg (a) \, \opeg (b) \, u -\opeg (a \, b) \, u \right\|_0 \le \dfrac{C}{\gamma} \, \| u \|_0 \, .
\end{equation*}
If $n \ge 3\, d +3$, then for another constant $C$, there holds
\begin{equation*}
\forall \, u \in {\mathcal S} (\R^{d+1}) \, ,\quad
\left\| \opeg (a) \, \opeg (b) \, u -\opeg (a \, b) \, u \right\|_{H^{1,\eps},\gamma} \le C \, \| u \|_0 \, ,
\end{equation*}
uniformly in $\eps$ and $\gamma$.

(b)\; Let $a \in S_n^1,b \in S_n^0$ or $a \in S_n^0,b \in S_n^1$, $n \ge 3\, d +4$, and in each case suppose 
$b=\sigma(\eps V,\X,\gamma)$ where  $V\in H^{s_0}(\R^{d+1})$ for some $s_0>\frac{d+1}{2}+1$. Then there 
exists a constant $C \ge 0$ such that for all $\eps \in \, ]0,1]$ and for all $\gamma \ge 1$, there holds
\begin{equation*}
\forall \, u \in {\mathcal S} (\R^{d+1}) \, ,\quad
\left\| \opeg (a) \, \opeg (b) \, u -\opeg (a \, b) \, u \right\|_0 \le C \, \| u \|_0 \, .
\end{equation*}
\end{prop}

\begin{prop}
\label{prop21}
Let $a \in S_n^{-1},b \in S_n^1$, $n \ge 3\, d +4$, and suppose $b=\sigma(\eps V,\X,\gamma)$ where $V 
\in H^{s_0}(\R^{d+1})$ for some $s_0>\frac{d+1}{2}+1$. Then $\opeg (a) \, \opeg (b)$ defines a bounded 
operator on $H^{1,\eps}$ and there exists a constant $C \ge 0$ such that for all $\eps \in \, ]0,1]$ and for 
all $\gamma \ge 1$, there holds
\begin{equation*}
\forall \, u \in {\mathcal S} (\R^{d+1}) \, ,\quad
\left\| \opeg (a) \, \opeg (b) \, u -\opeg (a \, b) \, u \right\|_{H^{1,\eps},\gamma} \le C \, \| u \|_0 \, .
\end{equation*}
\end{prop}

\noindent Our final result is G{\aa}rding's inequality.

\begin{theo}
\label{thm11}
Let $\sigma \in {\bf S}^0$ satisfy $\text{\rm Re} \, \sigma (v,\xi,\gamma) \ge C_K>0$ for all $v$ in a compact 
subset $K$ of ${\mathcal O}$. Let now $a \in S_0^n$, $n \ge 2\, d+2$ be given by \eqref{singularsymbolp}, 
where $V\in H^{s_0}(\R^{d+1})$ for some $s_0>\frac{d+1}{2}+1$ and is valued in a convex compact subset 
$K$. Then for all $\delta >0$, there exists $\gamma_0$ which depends uniformly on $V$, the constant $C_K$ 
and $\delta$, such that for all $\gamma \ge \gamma_0$ and all $u \in {\mathcal S}(\R^{d+1})$, there holds
\begin{equation*}
\text{\rm Re } \langle \opeg (a) \, u ;u \rangle_{L^2} \ge (C_K-\delta) \, \| u \|_0^2 \, .
\end{equation*}
\end{theo}

\subsection{Extended calculus}
\label{extended}

\emph{\quad} In the proof of Corollary \ref{estimH1} we use a slight extension of the singular calculus. For 
given parameters $0<\delta_1<\delta_2<1$, we choose a cutoff $\chi^e (\xi',\frac{k\, \beta}{\eps},\gamma)$ 
such that
\begin{align*}
\begin{split}
&0\leq \chi^e \leq 1\, ,\\
&\chi^e \left( \xi',\dfrac{k\, \beta}{\eps},\gamma \right) =1 \text{ on } \left\{
(\gamma^2 +|\xi'|^2)^{1/2} \leq \delta_1 \, \left| \dfrac{k\, \beta}{\eps} \right| \right\} \, ,\\
&\mathrm{supp }\chi^e \subset \left\{ (\gamma^2 +|\xi'|^2)^{1/2} \leq \delta_2 \, \left| \dfrac{k\, \beta}{\eps} \right|
\right\} \, ,
\end{split}
\end{align*}
and define a corresponding Fourier multiplier $\chi_D$ in the extended calculus by the formula \eqref{singularpseudop} 
with $\chi^e (\xi',\frac{k\, \beta}{\eps},\gamma)$ in place of $\sigma(\eps V,X,\gamma)$. Composition laws involving 
such operators are proved in \cite{CGW2}, but here we need only the fact that part {\it (a)} of Proposition \ref{prop20} 
holds when either $a$ or $b$ is replaced by an extended cutoff $\chi^e$.

\bibliographystyle{alpha}
\bibliography{Pulses2}
\end{document}